\documentclass[12pt,reqno,twoside]{amsart}%

\usepackage{amsmath,amsthm,amscd,amsthm,upref,indentfirst}
\usepackage{amsfonts,mathrsfs}

\usepackage{amssymb,amsbsy,bm}
\usepackage{times}
\usepackage{amsaddr}
\usepackage{soul}
\usepackage{slashbox}

\usepackage{mathabx}

\usepackage[usenames,dvipsnames]{color}

\theoremstyle{plain}
\newtheorem{theorem}{Theorem}
\newtheorem{assertion}[theorem]{Assertion}

\newtheorem{proposition}[theorem]{Proposition}

\theoremstyle{definition}
\newtheorem{definition}[theorem]{Definition}

\theoremstyle{remark}
\newtheorem{remark}[theorem]{Remark}

\newtheorem{example}[theorem]{Example}
\numberwithin{equation}{section}
\numberwithin{theorem}{section}

\usepackage[scaled=0.86]{helvet}
\renewcommand{\mathfrak}[1]{{\textbf{\upshape #1}}}
\renewcommand{\mathbf}{\bm}
\renewcommand{\emph}[1]{\textrm{{\upshape #1}}}
\usepackage[scr=euler,scrscaled=1.15,bb=txof,bbscaled=1.16,cal=boondox,calscaled=1.15]{mathalfa}
\renewcommand{\mathit}[1]{\mathscr #1}
\usepackage{exscale}

\numberwithin{equation}{section}
\numberwithin{theorem}{section}

\usepackage[normalem]{ulem}
\usepackage[square,comma]{natbib}

\usepackage{mparhack}

\usepackage{keyval}
\usepackage{totcount}
\newtotcounter{citnum} %From the package documentation
\def\oldbibitem{} \let\oldbibitem=\bibitem
\def\bibitem{\stepcounter{citnum}\oldbibitem}

\usepackage[verbose]{backref}
\backrefsetup{verbose=false}
\renewcommand*{\backref}[1]{}
\renewcommand*{\backrefalt}[4]{[{\tiny%
    \ifcase #1 \textsl{Not cited}%
          \or \textsl{Cited on page}~\textcolor{BrickRed}{#2}%
          \else \textsl{Cited on pages}~\textcolor{BrickRed}{#2}%
    \fi%
    }]}

\usepackage{fancyhdr}
\author{\small\scshape S\lowercase{teven} D\lowercase{uplij}}

\address{% \small \scshape
Center for Information Technology (WWU IT),
Universit\"at M\"unster,
R\"ontgenstrasse 7-13\\
D-48149 M\"unster,
Germany}
\email{\small \sf douplii@uni-muenster.de;
sduplij@gmail.com;
https://ivv5hpp.uni-muenster.de/u/douplii}
\title{\large\bfseries\scshape
H\lowercase{igher braid groups and regular semigroups
from polyadic-binary correspondence
}}
\date{\textit{of start} January 8, 2021. \textit{Date}:
\textit{of completion} February 7, 2021.
\newline
\mbox{}\hskip 1.16em
\textit{Total}:
27
references%,
}

\renewcommand{\refname}{\textsc{References}}

\let\origsection\section
\renewcommand{\section}[1]{\sectionmark{#1}\origsection{#1}}
\let\origsubsection\subsection
\renewcommand{\subsection}[1]{\subsectionmark{#1}\origsubsection{#1}}

\makeatletter
\renewenvironment{thebibliography}[1]{%
  \@xp\origsection\@xp*\@xp{\refname}%
  \normalfont\footnotesize\labelsep .9em\relax
  \renewcommand\theenumiv{\arabic{enumiv}}\let\p@enumiv\@empty
  \vspace*{-5pt}% NEW
  \list{\@biblabel{\theenumiv}}{\settowidth\labelwidth{\@biblabel{#1}}%
    \leftmargin\labelwidth \advance\leftmargin\labelsep
    \usecounter{enumiv}}%
  \sloppy \clubpenalty\@M \widowpenalty\clubpenalty
  \sfcode`\.=\@m
}{%
  \def\@noitemerr{\@latex@warning{Empty `thebibliography' environment}}%
  \endlist
}
\makeatother

\subjclass[2010]{16T25, 17A42, 20B30, 20F36, 20M17, 20N15}
\keywords{regular semigroup, braid group, generator, relation, presentation, Coxeter group, symmetric group, polyadic matrix group, querelement, idempotence, finite order element}

\begin{document}
\mbox{}

\begin{abstract}
\noindent In this note we first consider a ternary matrix group related to the
von Neumann regular semigroups and to the Artin braid group (in an algebraic
way). The product of a special kind of ternary matrices (idempotent and of
finite order) reproduces the regular semigroups and braid groups with their
binary multiplication of components. We then generalize the construction to
the higher arity case, which allows us to obtain some higher degree versions
(in our sense) of the regular semigroups and braid groups. The latter are
connected with the generalized polyadic braid equation and $R$-matrix
introduced by the author, which differ from any version of the well-known
tetrahedron equation and higher-dimensional analogs of the Yang-Baxter
equation, $n$-simplex equations. The higher degree (in our sense) Coxeter group
and symmetry groups are then defined, and it is shown that these are connected
only in the non-higher case.
\end{abstract}
\maketitle

\thispagestyle{empty}

\mbox{}
\tableofcontents

\newpage

\pagestyle{fancy}

\addtolength{\footskip}{15pt}

\renewcommand{\sectionmark}[1]{%
\markboth{
{ \scshape #1}}{}}

\renewcommand{\subsectionmark}[1]{%
\markright{
\mbox{\;}\\[5pt]
\textmd{#1}}{}}

\fancyhead{}
\fancyhead[EL,OR]{\leftmark}
\fancyhead[ER,OL]{\rightmark}
\fancyfoot[C]{\scshape -- \textcolor{BrickRed}{\thepage} --}
\fancyfoot[R]{}

\renewcommand\headrulewidth{0.5pt}
\fancypagestyle {plain1}{ %
\fancyhf{}
\renewcommand {\headrulewidth }{0pt}
\renewcommand {\footrulewidth }{0pt}
}

\fancypagestyle{plain}{ %
\fancyhf{}
\fancyhead[C]{\scshape S\lowercase{teven} D\lowercase{uplij} \hskip 0.7cm \MakeUppercase{Higher analogs of regular semigroups and braid groups from polyadic-binary correspondence}}
\fancyfoot[C]{\scshape - \thepage  -}
\renewcommand {\headrulewidth }{0pt}
\renewcommand {\footrulewidth }{0pt}
}

\fancypagestyle{fancyref}{ %
\fancyhf{} % remove everything
\fancyhead[C]{\scshape R\lowercase{eferences} }
\fancyfoot[C]{\scshape -- \textcolor{BrickRed}{\thepage} --}
\renewcommand {\headrulewidth }{0.5pt}
\renewcommand {\footrulewidth }{0pt}
}

\fancypagestyle{emptyf}{
\fancyhead{}
\fancyfoot[C]{\scshape -- \textcolor{BrickRed}{\thepage} --}
\renewcommand{\headrulewidth}{0pt}
}
\mbox{}
\vskip 1.5cm
\thispagestyle{emptyf}

\section{\textsc{Introduction}}

We begin by observing that the defining relations of the von Neumann regular
semigroups (e.g. \cite{grillet,howie,petrich3}) and the Artin braid group
\cite{kas/tur2008,kauffman} correspond to such properties of ternary matrices
(over the same set) as idempotence and the orders of elements (period). We
then generalize the correspondence thus introduced to the polyadic case and
thereby obtain higher degree (in our definition) analogs of the former. The
higher (degree) regular semigroups obtained in this way have appeared
previously in semisupermanifold theory \cite{duplij}, and higher regular
categories in TQFT \cite{dup/mar7}. The representations of the higher braid
relations in vector spaces coincide with the higher braid equation and
corresponding generalized $R$-matrix obtained in \cite{dup2018d}, as do the
ordinary braid group and the Yang-Baxter equation \cite{tur88}. The proposed
constructions use polyadic group methods and differ from the tetrahedron
equation \cite{zam81} and $n$-simplex equations \cite{hie} connected with
the braid group representations \cite{li/hu,hu97}, also from higher braid groups
of \cite{man/sch89}. Finally, we define higher degree (in our sense)
versions of the Coxeter group and the symmetric group and show that they are
connected in the classical (i.e. non-higher) case only.

\section{\textsc{Preliminaries}}

There is a general observation \cite{nik84}, that a block-matrix, forming a
semisimple $\left(  2,k\right)  $-ring (Artinian ring with binary addition and
$k$-ary multiplication) has the shape%
\begin{equation}
\mathsf{M}\left(  k-1\right)  \equiv\mathsf{M}^{\left(  \left(  k-1\right)
\times\left(  k-1\right)  \right)  }=\left(
\begin{array}
[c]{ccccc}%
0 & \mathsf{m}^{\left(  i_{1}\times i_{2}\right)  } & 0 & \ldots & 0\\
0 & 0 & \mathsf{m}^{\left(  i_{2}\times i_{3}\right)  } & \ldots & 0\\
0 & 0 & \vdots & \ddots & \vdots\\
\vdots & \vdots & 0 & \ddots & \mathsf{m}^{\left(  i_{k-2}\times
i_{k-1}\right)  }\\
\mathsf{m}^{\left(  i_{k-1}\times i_{1}\right)  } & 0 & 0 & \ldots & 0
\end{array}
\right)  . \label{m}%
\end{equation}
In other words, it is given by the cyclic shift $\left(  k-1\right)
\times\left(  k-1\right)  $ matrix, in which identities are replaced by blocks
of suitable sizes and with arbitrary entries.

The set $\left\{  \mathsf{M}\left(  k-1\right)  \right\}  $ is closed with
respect to the product of $k$ matrices, and we will therefore call them
$k$-ary matrices. They form a $k$-ary semigroup, and when the blocks are over
an associative binary ring, then total associativity follows from the
associativity of matrix multiplication.

Our proposal is to use single arbitrary elements (from rings with associative
multiplication) in place of the blocks $\mathsf{m}^{\left(  i\times j\right)
}$, supposing that the elements of the multiplicative part $G$ of the rings
form binary (semi)groups having some special properties. Then we investigate
the similar correspondence between the (multiplicative) properties of the
matrices $\mathsf{M}^{\left(  k-1\right)  }$, related to idempotence and
order, and the appearance of the relations in $G$ leading to regular
semigroups and braid groups, respectively. We call this connection a polyadic
matrix-binary (semi)group correspondence (or in short the
\textit{polyadic-binary correspondence}).

In the lowest -arity case $k=3$, the ternary case, the $2\times2$ matrices
$\mathsf{M}^{\left(  2\right)  }$ are anti-triangle. From $\left(
\mathsf{M}\left(  2\right)  \right)  ^{3}=\mathsf{M}\left(  2\right)  $ and
$\left(  \mathsf{M}\left(  2\right)  \right)  ^{3}\sim\mathsf{E}\left(
2\right)  $ (where $\mathsf{E}\left(  2\right)  $ is the ternary identity, see
below), we obtain the correspondences of the above conditions on
$\mathsf{M}\left(  2\right)  $ with the ordinary regular semigroups and braid
groups, respectively. In this way we extend the polyadic-binary correspondence
on -arities $k\geq4$ to get the higher relations%
\begin{equation}
\left(  \mathsf{M}\left(  k-1\right)  \right)  ^{k}=\left\{
\begin{array}
[c]{l}%
=\mathsf{M}\left(  k-1\right)  \ \ \ \text{corresponds to higher
}k\text{-degree regular semigroups},\\
=q\mathsf{E}\left(  k-1\right)  \ \ \ \text{corresponds to higher
}k\text{-degree braid groups},
\end{array}
\right.  \label{mk}%
\end{equation}
where $\mathsf{E}\left(  k-1\right)  $ is the $k$-ary identity (see below) and
$q$ is a fixed element of the braid group.

\section{\textsc{Ternary matrix group corresponding to the regular semigroup}}

Let $G_{free}=\left\{  \mathrm{G}\mid\mu_{2}^{g}\right\}  $ be a free
semigroup with the underlying set $\mathrm{G}=\left\{  g^{\left(  i\right)
}\right\}  $ and the binary multiplication. The anti-diagonal matrices over
$G_{free}$
\begin{equation}
M^{g}\left(  2\right)  =M^{\left(  2\times2\right)  }\left(  g^{\left(
1\right)  },g^{\left(  2\right)  }\right)  =\left(
\begin{array}
[c]{cc}%
0 & g^{\left(  1\right)  }\\
g^{\left(  2\right)  } & 0
\end{array}
\right)  ,\ \ \ g^{\left(  1\right)  },g^{\left(  2\right)  }\in G_{free}
\label{mg}%
\end{equation}
form a ternary semigroup $\mathcal{M}_{3}^{g}\equiv\mathcal{M}_{k=3}%
^{g}=\left\{  \mathrm{M}^{g}\left(  2\right)  \mid\mu_{3}^{g}\right\}  $,
where $\mathrm{M}^{g}\left(  2\right)  =\left\{  M^{g}\left(  2\right)
\right\}  $ is the set of ternary matrices (\ref{mg}) closed under the ternary
multiplication%
\begin{equation}
\mu_{3}^{g}\left[  M_{1}^{g}\left(  2\right)  ,M_{2}^{g}\left(  2\right)
,M_{3}^{g}\left(  2\right)  \right]  =M_{1}^{g}\left(  2\right)  M_{2}%
^{g}\left(  2\right)  M_{3}^{g}\left(  2\right)  ,\ \ \ \forall M_{1}%
^{g}\left(  2\right)  ,M_{2}^{g}\left(  2\right)  ,M_{3}^{g}\left(  2\right)
\in\mathcal{M}_{3}^{g}, \label{m3}%
\end{equation}
being the ordinary matrix product. Recall that an element $M^{g}\left(
2\right)  \in\mathcal{M}_{3}^{g}$ is idempotent, if%
\begin{equation}
\mu_{3}^{g}\left[  M^{g}\left(  2\right)  ,M^{g}\left(  2\right)
,M^{g}\left(  2\right)  \right]  =M^{g}\left(  2\right)  ,
\end{equation}
which in the matrix form (\ref{m3}) leads to%
\begin{equation}
\left(  M^{g}\left(  2\right)  \right)  ^{3}=M^{g}\left(  2\right)  .
\label{mg3}%
\end{equation}

We denote the set of idempotent ternary matrices by $\mathrm{M}%
_{\operatorname*{id}}^{g}\left(  2\right)  =\left\{  M_{\operatorname*{id}%
}^{g}\left(  2\right)  \right\}  $.

\begin{definition}
A ternary matrix semigroup in which every element is idempotent (\ref{mg3}) is
called an \textit{idempotent ternary semigroup}.
\end{definition}

Using (\ref{mg}) and (\ref{mg3}) the idempotence expressed in components gives
the regularity conditions%
\begin{align}
g^{\left(  1\right)  }g^{\left(  2\right)  }g^{\left(  1\right)  }  &
=g^{\left(  1\right)  },\label{g1}\\
g^{\left(  2\right)  }g^{\left(  1\right)  }g^{\left(  2\right)  }  &
=g^{\left(  2\right)  },\ \ \ \forall g^{\left(  1\right)  },g^{\left(
2\right)  }\in G_{free}. \label{g2}%
\end{align}

\begin{definition}
A binary semigroup $G_{free}$ in which any two elements are mutually regular
(\ref{g1})--(\ref{g2}) is called a \textit{regular semigroup} $G_{reg}$.
\end{definition}

\begin{proposition}
\label{prop-smg-ab}The set of idempotent ternary matrices (\ref{mg3}) form a
ternary semigroup \textit{ }$\mathcal{M}_{3,\operatorname*{id}}^{g}=\left\{
\mathrm{M}_{\operatorname*{id}}\left(  2\right)  \mid\mu_{3}\right\}  $, if
$G_{reg}$ is abelian.
\end{proposition}

\begin{proof}
It follows from (\ref{g1})--(\ref{g2}), that idempotence (and following from
it regularity) is preserved with respect to the ternary multiplication
(\ref{m3}), only when any $g^{\left(  1\right)  },g^{\left(  2\right)  }\in
G_{free}$ commute.
\end{proof}

\begin{definition}
We say that the set of idempotent ternary matrices $\mathrm{M}%
_{\operatorname*{id}}^{g}\left(  2\right)  $ (\ref{mg3}) is in
\textit{ternary-binary correspondence} with the regular (binary) semigroup
$G_{reg}$ and write this as%
\begin{equation}
\mathrm{M}_{\operatorname*{id}}^{g}\left(  2\right)  \Bumpeq G_{reg}\label{c}%
\end{equation}

\end{definition}

This means that such property of the ternary matrices as their idempotence
(\ref{mg3}) leads to the regularity conditions (\ref{g1})--(\ref{g2}) in the
correspondent binary group $G_{free}$.

\begin{remark}
The correspondence (\ref{c}) is not a homomorphism and not a bi-element
mapping \cite{bor/dud/dup3}, and also not a heteromorphism in sense of
\cite{dup2018a}, because we do not demand that the set of idempotent matrices
$M_{\operatorname*{id}}^{g}\left(  2\right)  $ form a ternary semigroup (which
is possible in commutative case of $G_{free}$ only, see \textbf{Proposition}
\ref{prop-smg-ab}).
\end{remark}

\section{\label{sec-smgk}\textsc{Polyadic matrix semigroup corresponding to
the higher regular semigroup}}

We next extend the ternary-binary correspondence (\ref{c}) to the $k$-ary
matrix case (\ref{m}) and thereby obtain higher $k$-regular binary
semigroups\footnote{We use the following notation:
\par
Round brackets: $\left(  k\right)  $ is size of matrix $k\times k$, also the
sequential number of a matrix element.
\par
Square brackets: $\left[  k\right]  $ is number of multipliers in the
regularity and braid conditions.
\par
Angle brackets: $\left\langle \ell\right\rangle _{k}$ is the polyadic power
(number of $k$-ary multiplications).
\par
{}}.

Let us introduce the $\left(  k-1\right)  \times\left(  k-1\right)  $ matrix
over a binary group $G_{free}$ of the form (\ref{m})%
\begin{equation}
M^{g}\left(  k-1\right)  \equiv M^{\left(  \left(  k-1\right)  \times\left(
k-1\right)  \right)  }\left(  g^{\left(  1\right)  },g^{\left(  2\right)
},\ldots,g^{\left(  k-1\right)  }\right)  =\left(
\begin{array}
[c]{ccccc}%
0 & g^{\left(  1\right)  } & 0 & \ldots & 0\\
0 & 0 & g^{\left(  1\right)  } & \ldots & 0\\
0 & 0 & \vdots & \ddots & \vdots\\
\vdots & \vdots & 0 & \ddots & g^{\left(  k-2\right)  }\\
g^{\left(  k-1\right)  } & 0 & 0 & \ldots & 0
\end{array}
\right)  , \label{mgk}%
\end{equation}
where $g^{\left(  i\right)  }\in G_{free}$.

\begin{definition}
The set of $k$-ary matrices $\mathrm{M}^{g}\left(  k-1\right)  $ (\ref{mgk})
over $G_{free}$ is a $k$\textit{-ary matrix semigroup} $\mathcal{M}_{k}%
^{g}=\left\{  \mathrm{M}^{g}\left(  k-1\right)  \mid\mu_{k}^{g}\right\}  $,
where the multiplication
\begin{align}
&  \mu_{k}^{g}\left[  M_{1}^{g}\left(  k-1\right)  ,M_{2}^{g}\left(
k-1\right)  ,\ldots,M_{k}^{g}\left(  k-1\right)  \right] \nonumber\\
&  =M_{1}^{g}\left(  k-1\right)  M_{2}^{g}\left(  k-1\right)  \ldots M_{k}%
^{g}\left(  k-1\right)  ,\ \ M_{i}^{g}\left(  k-1\right)  \in\mathcal{M}%
_{k}^{g} \label{mgm}%
\end{align}
is the ordinary product of $k$ matrices $M_{i}^{g}\left(  k-1\right)  \equiv
M^{\left(  \left(  k-1\right)  \times\left(  k-1\right)  \right)  }\left(
g_{i}^{\left(  1\right)  },g_{i}^{\left(  2\right)  },\ldots,g_{i}^{\left(
k-1\right)  }\right)  $, see (\ref{mgk}).
\end{definition}

Recall that the polyadic power $\ell$ of an element $M$ from a $k$-ary
semigroup $\mathcal{M}_{k}$ is defined by (e.g. \cite{pos})%
\begin{equation}
M^{\left\langle \ell\right\rangle _{k}}=\left(  \mu_{k}\right)  ^{\ell}\left[
\overset{\ell\left(  k-1\right)  +1}{\overbrace{M,\ldots,M}}\right]  ,
\label{ml}%
\end{equation}
such that $\ell$ coincides with the \textit{number of }$k$\textit{-ary
multiplications}. In the binary case $k=2$ the polyadic power is connected
with the ordinary power $p$ (\textit{number of elements} in the product) as
$p=\ell+1$, i.e. $M^{\left\langle \ell\right\rangle _{2}}=M^{\ell+1}=M^{p}$.
In the ternary case $k=3$ we have $\left\langle \ell\right\rangle _{3}%
=2\ell+1$, and so the l.h.s. of (\ref{mg3}) is of polyadic power $\ell=1$.

\begin{definition}
An element of a $k$-ary semigroup $M\in\mathcal{M}_{3}$ is called
\textit{idempotent}, if its first polyadic power coincides with itself%
\begin{equation}
M^{\left\langle 1\right\rangle _{k}}=M, \label{m1m}%
\end{equation}
and $\left\langle \ell\right\rangle $\textit{-idempotent}, if%
\begin{equation}
M^{\left\langle \ell\right\rangle _{k}}=M,\ \ M^{\left\langle \ell
-1\right\rangle _{k}}\neq M. \label{mlm}%
\end{equation}

\end{definition}

\begin{definition}
A $k$-ary semigroup$\ \mathcal{M}_{k}$ is called \textit{idempotent} ($\ell
$\textit{-idempotent}), if each of its elements $M\in\mathcal{M}_{k}$ is
idempotent ($\left\langle \ell\right\rangle $\textit{-idempotent}).
\end{definition}

\begin{assertion}
\label{asser-lreg}From $M^{\left\langle 1\right\rangle _{k}}=M$ it follows
that $M^{\left\langle \ell\right\rangle _{k}}=M$, but not vice-versa,
therefore all $\left\langle 1\right\rangle $-idempotent elements are
$\left\langle \ell\right\rangle $-idempotent, but an $\left\langle
\ell\right\rangle $-idempotent element need not be $\left\langle
1\right\rangle $-idempotent.
\end{assertion}

Therefore, the definition given in(\ref{mlm}) makes sense.

\begin{proposition}
If a $k$-ary matrix $M^{g}\left(  k-1\right)  \in\mathcal{M}_{k}^{g}$ is
idempotent (\ref{m1m}), then its elements satisfy the $\left(  k-1\right)  $
relations%
\begin{align}
g^{\left(  1\right)  }g^{\left(  2\right)  },\ldots g^{\left(  k-2\right)
}g^{\left(  k-1\right)  }g^{\left(  1\right)  }  &  =g^{\left(  1\right)
},\label{gg1}\\
g^{\left(  2\right)  }g^{\left(  3\right)  },\ldots g^{\left(  k-1\right)
}g^{\left(  1\right)  }g^{\left(  2\right)  }  &  =g^{\left(  2\right)
},\label{gg2}\\
&  \vdots\nonumber\\
g^{\left(  k-1\right)  }g^{\left(  1\right)  }g^{\left(  2\right)  },\ldots
g^{\left(  k-2\right)  }g^{\left(  k-1\right)  }  &  =g^{\left(  k-1\right)
},\ \ \ \forall g^{\left(  1\right)  },\ldots,g^{\left(  k-1\right)  }\in
G_{free}. \label{gg3}%
\end{align}

\end{proposition}

\begin{proof}
This follows from (\ref{mgk}), (\ref{mgm}) and (\ref{m1m}).
\end{proof}

\begin{definition}
The relations (\ref{gg1})--(\ref{gg3}) are called (higher) $\left[  k\right]
$\textit{-regularity} (or \textit{higher }$k$\textit{-degree regularity}). The
case $k=3$ is the standard regularity ($\left[  3\right]  $-regularity in our
notation) (\ref{g1})--(\ref{g2}).
\end{definition}

\begin{proposition}
If a $k$-ary matrix $M^{g}\left(  k-1\right)  \in\mathcal{M}_{k}^{g}$ is
$\left\langle \ell\right\rangle $-idempotent (\ref{mlm}), then its elements
satisfy the following $\left(  k-1\right)  $ relations%
\begin{align}
\overset{\ell}{\overbrace{\left(  g^{\left(  1\right)  }g^{\left(  2\right)
}\ldots g^{\left(  k-2\right)  }g^{\left(  k-1\right)  }\right)  \ldots\left(
g^{\left(  1\right)  }g^{\left(  2\right)  },\ldots g^{\left(  k-2\right)
}g^{\left(  k-1\right)  }\right)  }}g^{\left(  1\right)  }  &  =g^{\left(
1\right)  },\label{ggg1}\\
\overset{\ell}{\overbrace{\left(  g^{\left(  2\right)  }g^{\left(  3\right)
},\ldots g^{\left(  k-2\right)  }g^{\left(  k-1\right)  }g^{\left(  1\right)
}\right)  \ldots\left(  g^{\left(  2\right)  }g^{\left(  3\right)  }\ldots
g^{\left(  k-2\right)  }g^{\left(  k-1\right)  }g^{\left(  1\right)  }\right)
}}g^{\left(  2\right)  }  &  =g^{\left(  2\right)  },\label{ggg2}\\
&  \vdots\nonumber\\
\overset{\ell}{\overbrace{\left(  g^{\left(  k-1\right)  }g^{\left(  1\right)
}g^{\left(  2\right)  }\ldots g^{\left(  k-3\right)  }g^{\left(  k-2\right)
}\right)  \ldots\left(  g^{\left(  k-1\right)  }g^{\left(  1\right)
}g^{\left(  2\right)  }\ldots g^{\left(  k-3\right)  }g^{\left(  k-2\right)
}\right)  }}g^{\left(  k-1\right)  }  &  =g^{\left(  k-1\right)
},\label{ggg3}\\
\ \ \ \forall g^{\left(  1\right)  },\ldots,g^{\left(  k-1\right)  }  &  \in
G_{free}.\nonumber
\end{align}

\end{proposition}

\begin{proof}
This also follows from (\ref{mgk}), (\ref{mgm}) and (\ref{mlm}).
\end{proof}

\begin{definition}
The relations (\ref{gg1})--(\ref{gg3}) are called (higher) $\left[  k\right]
$\textit{-}$\left\langle \ell\right\rangle $-\textit{regularity}. The case
$k=3$ (\ref{g1})--(\ref{g2}) is the standard regularity ($\left[  3\right]
$-$\left\langle 1\right\rangle $-regularity in this notation).
\end{definition}

\begin{definition}
A binary semigroup $G_{free}$, in which any $k-1$ elements are $\left[
k\right]  $-regular ($\left[  k\right]  $-$\left\langle \ell\right\rangle
$-regular), is called a \textit{higher }$\left[  k\right]  $\textit{-regular
(}$\left[  k\right]  $\textit{-}$\left\langle \ell\right\rangle $%
\textit{-regular) semigroup} $G_{reg}\left[  k\right]  $ ($G_{\left\langle
\ell\right\rangle \text{-}reg}\left[  k\right]  $).
\end{definition}

Similarly to \textbf{Assertion} \ref{asser-lreg}, it is seen that $\left[
k\right]  $-$\left\langle \ell\right\rangle $-regularity (\ref{ggg1}%
)--(\ref{ggg3}) follows from $\left[  k\right]  $-regularity (\ref{gg1}%
)--(\ref{gg3}), but not the other way around, and therefore we have

\begin{assertion}
If a binary semigroup $G_{reg}\left[  k\right]  $ is $\left[  k\right]
$-regular, then it is $\left[  k\right]  $-$\left\langle \ell\right\rangle
$-regular as well, but not vice-versa.
\end{assertion}

\begin{proposition}
\label{prop-smg-abk}The set of idempotent ($\left\langle \ell\right\rangle
$-idempotent) $k$-ary matrices $\mathrm{M}_{\operatorname*{id}}^{g}\left(
k-1\right)  $ form a $k$-ary semigroup \textit{ }$\mathcal{M}%
_{3,\operatorname*{id}}^{g}=\left\{  \mathrm{M}_{\operatorname*{id}}%
^{g}\left(  k-1\right)  \mid\mu_{k}^{g}\right\}  $, if and only if
$G_{reg}\left[  k\right]  $ ($G_{\ell\text{-}reg}\left[  k\right]  $) is abelian.
\end{proposition}

\begin{proof}
It follows from (\ref{gg1})--(\ref{ggg3}) that the idempotence ($\left\langle
\ell\right\rangle $-idempotence) and the following $\left[  k\right]
$-regularity ($\left[  k\right]  $-$\left\langle \ell\right\rangle
$-regularity) are preserved with respect the $k$-ary multiplication
(\ref{mgm}) only in the case, when all $g^{\left(  1\right)  },\ldots
,g^{\left(  k-1\right)  }\in G_{free}$ mutually commute.
\end{proof}

By analogy with (\ref{c}), we have

\begin{definition}
We will say that the set of $k$-ary $\left(  k-1\right)  \times\left(
k-1\right)  $ matrices $\mathrm{M}_{\operatorname*{id}}^{g}\left(  k-1\right)
$ (\ref{mgk}) over the underlying set $\mathrm{G}$ is in
\textit{polyadic-binary correspondence} with the binary $\left[  k\right]
$-regular semigroup $G_{reg}\left[  k\right]  $ and write this as%
\begin{equation}
\mathrm{M}_{\operatorname*{id}}^{g}\left(  k-1\right)  \Bumpeq G_{reg}\left[
k\right]  . \label{ck}%
\end{equation}

\end{definition}

Thus, using the idempotence condition for $k$-ary matrices in components
(being simultaneously elements of a binary semigroup $G_{free}$) and the
polyadic-binary correspondence (\ref{ck}) we obtain the higher regularity
conditions (\ref{gg1})--(\ref{ggg3}) generalizing the ordinary regularity
(\ref{g1})--(\ref{g2}), which allows us to define the higher $\left[
k\right]  $-regular binary semigroups $G_{reg}\left[  k\right]  $
($G_{\left\langle \ell\right\rangle \text{-}reg}\left[  k\right]  $).

\begin{example}
\label{ex-smgk4}The lowest nontrivial ($k\geq3$) case is $k=4$, where the
$3\times3$ matrices over $G_{free}$ are of the shape%
\begin{equation}
M\left(  3\right)  =M^{\left(  3\times3\right)  }=\left(
\begin{array}
[c]{ccc}%
0 & a & 0\\
0 & 0 & b\\
c & 0 & 0
\end{array}
\right)  ,\ \ \ a,b,c\in G_{free},
\end{equation}
and they form the $4$-ary matrix semigroup $\mathcal{M}_{4}^{g}$. The
idempotence $\left(  M\left(  3\right)  \right)  ^{\left\langle 1\right\rangle
_{4}}=\left(  M\left(  3\right)  \right)  ^{3}=M\left(  3\right)  $ gives
three $\left[  4\right]  $-regularity conditions%
\begin{align}
abca  &  =a,\label{a1}\\
bcab  &  =b,\label{a2}\\
cabc  &  =c. \label{a3}%
\end{align}

According to the polyadic-binary correspondence (\ref{ck}), the conditions
(\ref{a1})--(\ref{a3}) are $\left[  4\right]  $-regularity relations for the
binary semigroup $G_{free}$, which defines to the higher $\left[  4\right]
$-regular binary semigroup $G_{reg}\left[  4\right]  $.

In the case $\ell=2$, we have $\left(  M\left(  3\right)  \right)
^{\left\langle 2\right\rangle _{4}}=\left(  M\left(  3\right)  \right)
^{7}=M\left(  3\right)  $, which gives three $\left[  4\right]  $%
-$\left\langle 2\right\rangle $-regularity conditions (they are different from
$\left[  7\right]  $-regularity)%
\begin{align}
abcabca  &  =a,\label{aa1}\\
bcabcab  &  =b,\label{aa2}\\
cabcabc  &  =c, \label{aa3}%
\end{align}
and these define the higher $\left[  4\right]  $-$\left\langle 2\right\rangle
$-regular binary semigroup $G_{\left\langle 2\right\rangle \text{-}reg}\left[
4\right]  $. Obviously, (\ref{aa1})--(\ref{aa3}) follow from (\ref{a1}%
)--(\ref{a3}), but not vice-versa.
\end{example}

The higher regularity conditions (\ref{a1})--(\ref{a3}) obtained above from
the idempotence of polyadic matrices using the polyadic-binary
correspondence,\textit{ }appeared first in \cite{dup18} and were then used for
transition functions in the investigation of semisupermanifolds \cite{duplij}
and higher regular categories in TQFT \cite{dup/mar2001a,dup/mar7}.

Now we turn to the second line of (\ref{mk}), and in the same way as above
introduce higher degree braid groups.

\section{\label{sec-tern}\textsc{Ternary matrix group corresponding to the
braid group}}

Recall the definition of the Artin braid group \cite{art47} in terms of
generators and relations \cite{kas/tur2008} (we follow the algebraic approach,
see, e.g. \cite{mar45}).

The \textit{Artin braid group} $B_{n}$ (with $n$ strands and the identity
$e\in B_{n}$) has the presentation by $n-1$ generators $\sigma_{1}%
,\ldots,\sigma_{n-1}$ satisfying $n\left(  n-1\right)  /2$ relations%
\begin{align}
\sigma_{i}\sigma_{i+1}\sigma_{i}  &  =\sigma_{i+1}\sigma_{i}\sigma
_{i+1},\ \ \ 1\leq i\leq n-2,\label{s1}\\
\sigma_{i}\sigma_{j}  &  =\sigma_{j}\sigma_{i},\ \ \ \ \left\vert
i-j\right\vert \geq2, \label{s2}%
\end{align}
where (\ref{s1}) are called the \textit{braid relations}, and (\ref{s2}) are
called \textit{far commutativity}. A general element of $B_{n}$ is a word of
the form%
\begin{equation}
w=\sigma_{i_{1}}^{p_{1}}\ldots\sigma_{i_{r}}^{p_{r}}\ldots\sigma_{i_{m}%
}^{p_{m}},\ \ \ i_{m}=1,\ldots,n, \label{w}%
\end{equation}
where $p_{r}\in\mathbb{Z}$ are (positive or negative) powers of the generators
$\sigma_{i_{r}}$, $r=1,\ldots,m$ and $m\in\mathbb{N}$.

For instance, $B_{3}$ is generated by $\sigma_{1}$ and $\sigma_{2}$ satisfying
one relation $\sigma_{1}\sigma_{2}\sigma_{1}=\sigma_{2}\sigma_{1}\sigma_{2}$,
and is isomorphic to the trefoil knot group. The group $B_{4}$ has 3
generators $\sigma_{1},\sigma_{2},\sigma_{3}$ satisfying%
\begin{align}
\sigma_{1}\sigma_{2}\sigma_{1}  &  =\sigma_{2}\sigma_{1}\sigma_{2}%
,\label{bs1}\\
\sigma_{2}\sigma_{3}\sigma_{2}  &  =\sigma_{3}\sigma_{2}\sigma_{3}%
,\label{bs2}\\
\sigma_{1}\sigma_{3}  &  =\sigma_{3}\sigma_{1}. \label{bs3}%
\end{align}

The representation theory of $B_{n}$ is well known and well established
\cite{kas/tur2008,kauffman}. The connections with the Yang-Baxter equation
were investigated, e.g. in \cite{tur88}.

Now we build a ternary group of matrices over $B_{n}$ having generators
satisfying relations which are connected with the braid relations
(\ref{s1})--(\ref{s2}). We then generalize our construction to a $k$-ary
matrix group, which gives us the possibility to \textquotedblleft go
back\textquotedblright\ and define some special higher analogs of the Artin
braid group.

Let us consider the set of anti-diagonal $2\times2$ matrices over $B_{n}$%
\begin{equation}
M\left(  2\right)  =M^{\left(  2\times2\right)  }\left(  b^{\left(  1\right)
},b^{\left(  2\right)  }\right)  =\left(
\begin{array}
[c]{cc}%
0 & b^{\left(  1\right)  }\\
b^{\left(  2\right)  } & 0
\end{array}
\right)  ,\ \ \ b^{\left(  1\right)  },b^{\left(  2\right)  }\in B_{n}.
\label{m2}%
\end{equation}

\begin{definition}
The set of matrices $\mathrm{M}\left(  2\right)  =\left\{  M\left(  2\right)
\right\}  $ (\ref{m2}) over $B_{n}$ form a \textit{ternary matrix semigroup}
$\mathcal{M}_{k=3}=\mathcal{M}_{3}=\left\{  \mathrm{M}\left(  2\right)
\mid\mu_{3}\right\}  $, where $k=3$ is the -arity of the following
multiplication%
\begin{align}
\mu_{3}\left[  M_{1}\left(  2\right)  ,M_{2}\left(  2\right)  ,M_{3}\left(
2\right)  \right]   &  \equiv M_{1}\left(  2\right)  ,M_{2}\left(  2\right)
,M_{3}\left(  2\right)  =M\left(  2\right)  ,\label{mb1}\\
b_{1}^{\left(  1\right)  }b_{2}^{\left(  2\right)  }b_{3}^{\left(  1\right)
}  &  =b^{\left(  1\right)  },\label{mb2}\\
b_{1}^{\left(  2\right)  }b_{2}^{\left(  1\right)  }b_{3}^{\left(  2\right)
}  &  =b^{\left(  2\right)  },\ \ \ \ b_{i}^{\left(  1\right)  }%
,b_{i}^{\left(  2\right)  }\in B_{n},\ \ M_{i}\left(  2\right)  =\left(
\begin{array}
[c]{cc}%
0 & b_{i}^{\left(  1\right)  }\\
b_{i}^{\left(  2\right)  } & 0
\end{array}
\right)  \label{mb3}%
\end{align}
and the associativity is governed by the associativity of both the ordinary
matrix product in the r.h.s. of (\ref{mb1}) and $B_{n}$.
\end{definition}

\begin{proposition}
$\mathcal{M}^{\left(  3\right)  }$ is a ternary matrix group.
\end{proposition}

\begin{proof}
Each element of the ternary matrix semigroup $M\left(  2\right)
\in\mathcal{M}_{3}$ is invertible (in the ternary sense) and has a
\textit{querelement} $\bar{M}\left(  2\right)  $ (a polyadic analog of the
group inverse \cite{dor3}) defined by%
\begin{equation}
\mu_{3}\left[  M\left(  2\right)  ,M\left(  2\right)  ,\bar{M}\left(
2\right)  \right]  =\mu_{3}\left[  M\left(  2\right)  ,\bar{M}\left(
2\right)  ,M\left(  2\right)  \right]  =\mu_{3}\left[  \bar{M}\left(
2\right)  ,M\left(  2\right)  ,M\left(  2\right)  \right]  =M\left(  2\right)
. \label{m1}%
\end{equation}
It follows from (\ref{mb1})--(\ref{mb3}), that%
\begin{equation}
\bar{M}\left(  2\right)  =\left(  M\left(  2\right)  \right)  ^{-1}=\left(
\begin{array}
[c]{cc}%
0 & \left(  b^{\left(  1\right)  }\right)  ^{-1}\\
\left(  b^{\left(  2\right)  }\right)  ^{-1} & 0
\end{array}
\right)  ,\ \ \ b^{\left(  1\right)  },b^{\left(  2\right)  }\in B_{n},
\label{mm}%
\end{equation}
where $\left(  M^{\left(  2\right)  }\right)  ^{-1}$ denotes the ordinary
matrix inverse (but not the binary group inverse which does not exist in the
$k$-ary case, $k\geq3$). Non-commutativity of $\mu_{3}$ is provided by
(\ref{mb2})--(\ref{mb3}).
\end{proof}

The ternary matrix group $\mathcal{M}_{3}$ has the \textit{ternary identity}%
\begin{equation}
E\left(  2\right)  =\left(
\begin{array}
[c]{cc}%
0 & e\\
e & 0
\end{array}
\right)  ,\ \ \ e\in B_{n}, \label{e2}%
\end{equation}
where $e$ is the identity of the binary group $B_{n}$, and%
\begin{equation}
\mu_{3}\left[  M\left(  2\right)  ,E\left(  2\right)  ,E\left(  2\right)
\right]  =\mu_{3}\left[  E\left(  2\right)  ,M\left(  2\right)  ,E\left(
2\right)  \right]  =\mu_{3}\left[  E\left(  2\right)  ,E\left(  2\right)
,M\left(  2\right)  \right]  =M\left(  2\right)  . \label{mee}%
\end{equation}

We observe that the ternary product $\mu_{3}$ in components is
\textquotedblleft naturally braided\textquotedblright\ (\ref{mb2}%
)--(\ref{mb3}). This allows us to ask the question: which generators of the
ternary group $\mathcal{M}_{3}$ can be constructed using the Artin braid group
generators $\sigma_{i}\in B_{n}$ and the relations (\ref{s1})--(\ref{s2})?

\section{\label{sec-terngen}\textsc{Ternary matrix generators}}

Let us introduce $\left(  n-1\right)  ^{2}$ ternary $2\times2$ \textit{matrix
generators}%
\begin{equation}
\Sigma_{ij}\left(  2\right)  =\Sigma_{ij}^{\left(  2\times2\right)  }\left(
\sigma_{i},\sigma_{j}\right)  =\left(
\begin{array}
[c]{cc}%
0 & \sigma_{i}\\
\sigma_{j} & 0
\end{array}
\right)  , \label{ss}%
\end{equation}
where $\sigma_{i}\in B_{n}$, $i=1\ldots,n-1$ are generators of the Artin braid
group. The querelement of $\Sigma_{ij}\left(  2\right)  $ is defined by
analogy with (\ref{mm}) as%
\begin{equation}
\bar{\Sigma}_{ij}\left(  2\right)  =\left(  \Sigma_{ij}\left(  2\right)
\right)  ^{-1}=\left(
\begin{array}
[c]{cc}%
0 & \sigma_{j}^{-1}\\
\sigma_{i}^{-1} & 0
\end{array}
\right)  .
\end{equation}

Now we are in a position to present a ternary matrix group with multiplication
$\mu_{3}$ in terms of generators and relations in such a way that the braid
group relations (\ref{s1})--(\ref{s2}) will be reproduced.

\begin{proposition}
The relations for the matrix generators $\Sigma_{ij}\left(  2\right)  $
corresponding to the braid group relations for $\sigma_{i}$ (\ref{s1}%
)--(\ref{s2}) have the form%
\begin{align}
& \mu_{3}\left[  \Sigma_{i,j+1}\left(  2\right)  ,\Sigma_{i,j+1}\left(
2\right)  ,\Sigma_{i,j+1}\left(  2\right)  \right]  =\mu_{3}\left[
\Sigma_{j+1,i}\left(  2\right)  ,\Sigma_{j+1,i}\left(  2\right)
,\Sigma_{j+1,i}\left(  2\right)  \right] \nonumber\\
&  =q_{i}^{\left[  3\right]  }E\left(  2\right)  ,\ \ \ 1\leq i\leq
n-2,\label{ms1}\\
& \mu_{3}\left[  \Sigma_{ij}\left(  2\right)  ,\Sigma_{ij}\left(  2\right)
,E\left(  2\right)  \right]  =\mu_{3}\left[  \Sigma_{ji}\left(  2\right)
,\Sigma_{ji}\left(  2\right)  ,E\left(  2\right)  \right]  ,\ \ \ \ \left\vert
i-j\right\vert \geq2, \label{ms3}%
\end{align}
where $q_{i}^{\left[  3\right]  }=\sigma_{i}\sigma_{i+1}\sigma_{i}%
=\sigma_{i+1}\sigma_{i}\sigma_{i+1}$, and $E\left(  2\right)  $ is the ternary
identity (\ref{e2}).
\end{proposition}

\begin{proof}
Use $\mu_{3}$ as the triple matrix product (\ref{mb1})--(\ref{mb3}) and the
braid relations (\ref{s1})--(\ref{s2}).
\end{proof}

\begin{definition}
We say that the ternary matrix group $\mathcal{M}_{3}^{gen\text{-}\Sigma}$
generated by the matrix generators $\Sigma_{ij}\left(  2\right)  $ satisfying
the relations (\ref{ms1})--(\ref{ms3}) is in \textit{ternary-binary
correspondence} with the braid (binary) group $B_{n}$, which is denoted as
(cf. (\ref{c}))%
\begin{equation}
\mathcal{M}_{3}^{gen\text{-}\Sigma}\Bumpeq B_{n}. \label{c3}%
\end{equation}

\end{definition}

Indeed, in components the relations (\ref{ms1}) give (\ref{s1}), and
(\ref{ms3}) leads to (\ref{s2}).

\begin{remark}
\label{rem-rep}Note that the above construction is totally different from the
bi-element representations of ternary groups considered in \cite{bor/dud/dup3}
(for $k$-ary groups see \cite{dup2018a}).
\end{remark}

\begin{definition}
An element $M\left(  2\right)  \in\mathcal{M}_{3}$ is of \textit{finite
polyadic (ternary) order}, if there exists a finite $\ell$ such that%
\begin{equation}
M\left(  2\right)  ^{\left\langle \ell\right\rangle _{3}}=M\left(  2\right)
^{2\ell+1}=E\left(  2\right)  , \label{me}%
\end{equation}
where $E\left(  2\right)  $ is the ternary matrix identity (\ref{e2}).
\end{definition}

\begin{definition}
An element $M\left(  2\right)  \in\mathcal{M}_{3}$ is of \textit{finite }%
$q$\textit{-polyadic (}$q$\textit{-ternary) order}, if there exists a finite
$\ell$ such that%
\begin{equation}
M\left(  2\right)  ^{\left\langle \ell\right\rangle _{3}}=M\left(  2\right)
^{2\ell+1}=qE\left(  2\right)  ,\ \ \ q\in B_{n}.
\end{equation}

\end{definition}

The relations (\ref{ms1}) therefore say that the ternary matrix generators
$\Sigma_{i,j+1}\left(  2\right)  $ are of finite $q$-ternary order. Each
element of $\mathcal{M}_{3}^{gen\text{-}\Sigma}$ is a ternary matrix word
(analogous to the binary word (\ref{w})), being the ternary product of the
polyadic powers (\ref{ml}) of the $2\times2$ matrix generators $\Sigma
_{ij}^{\left(  2\right)  }$ and their querelements $\bar{\Sigma}_{ij}^{\left(
2\right)  }$ (on choosing the first or second row)%
\begin{align}
W  &  =\left(
\begin{array}
[c]{c}%
\Sigma_{i_{1}j_{1}}\left(  2\right) \\
\bar{\Sigma}_{i_{1}j_{1}}\left(  2\right)
\end{array}
\right)  ^{\left\langle \ell_{1}\right\rangle _{3}},\ldots,\left(
\begin{array}
[c]{c}%
\Sigma_{i_{r}j_{r}}\left(  2\right) \\
\bar{\Sigma}_{i_{r}j_{r}}\left(  2\right)
\end{array}
\right)  ^{\left\langle \ell_{r}\right\rangle _{3}},\ldots,\left(
\begin{array}
[c]{c}%
\Sigma_{i_{m}j_{m}}\left(  2\right) \\
\bar{\Sigma}_{i_{m}j_{m}}\left(  2\right)
\end{array}
\right)  ^{\left\langle \ell_{m}\right\rangle _{3}}\nonumber\\
&  =\left(
\begin{array}
[c]{c}%
\Sigma_{i_{1}j_{1}}\left(  2\right) \\
\bar{\Sigma}_{i_{1}j_{1}}\left(  2\right)
\end{array}
\right)  ^{2\ell_{1}+1},\ldots,\left(
\begin{array}
[c]{c}%
\Sigma_{i_{r}j_{r}}\left(  2\right) \\
\bar{\Sigma}_{i_{r}j_{r}}\left(  2\right)
\end{array}
\right)  ^{2\ell_{r}+1},\ldots,\left(
\begin{array}
[c]{c}%
\Sigma_{i_{m}j_{m}}\left(  2\right) \\
\bar{\Sigma}_{i_{m}j_{m}}\left(  2\right)
\end{array}
\right)  ^{2\ell_{m}+1}, \label{ww}%
\end{align}
where $r=1,\ldots,m$, $i_{r},j_{r}=1,\ldots,n$ (from $B_{n}$), $\ell_{r}%
,m\in\mathbb{N}$. In the ternary case the total number of multipliers in
(\ref{ww}) should be compatible with (\ref{ml}), i.e. $\left(  2\ell
_{1}+1\right)  +\ldots+\left(  2\ell_{r}+1\right)  +\ldots+\left(  2\ell
_{m}+1\right)  =2\ell_{W}+1$, $\ell_{W}\in\mathbb{N}$, and $m$ is therefore
odd. Thus, we have

\begin{remark}
The ternary words (\ref{ww}) in components give only a subset of the binary
words (\ref{w}), and so $\mathcal{M}_{3}^{gen\text{-}\Sigma}$ corresponds to
$B_{n}$, but does not present it.
\end{remark}

\begin{example}
For $B_{3}$ we have only two ternary $2\times2$ matrix generators%
\begin{equation}
\Sigma_{12}\left(  2\right)  =\left(
\begin{array}
[c]{cc}%
0 & \sigma_{1}\\
\sigma_{2} & 0
\end{array}
\right)  ,\ \ \ \Sigma_{21}\left(  2\right)  =\left(
\begin{array}
[c]{cc}%
0 & \sigma_{2}\\
\sigma_{1} & 0
\end{array}
\right)  ,
\end{equation}
satisfying%
\begin{align}
\left(  \Sigma_{12}\left(  2\right)  \right)  ^{\left\langle 1\right\rangle
_{3}}  &  =\left(  \Sigma_{12}\left(  2\right)  \right)  ^{3}=q_{1}^{\left[
3\right]  }E\left(  2\right)  ,\label{c1}\\
\left(  \Sigma_{21}\left(  2\right)  \right)  ^{\left\langle 1\right\rangle
_{3}}  &  =\left(  \Sigma_{21}\left(  2\right)  \right)  ^{3}=q_{1}^{\left[
3\right]  }E\left(  2\right)  , \label{c2}%
\end{align}
where $q_{1}^{\left[  3\right]  }=\sigma_{1}\sigma_{2}\sigma_{1}=\sigma
_{2}\sigma_{1}\sigma_{2}$, and both matrix relations (\ref{c1})--(\ref{c2})
coincide in components.
\end{example}

\begin{example}
For $B_{4}$, the ternary matrix group $\mathcal{M}_{3}^{gen\text{-}\Sigma}$ is
generated by more generators satisfying the relations%
\begin{align}
\left(  \Sigma_{12}\left(  2\right)  \right)  ^{3}  &  =q_{1}^{\left(
3\right)  }E\left(  2\right)  ,\\
\left(  \Sigma_{21}\left(  2\right)  \right)  ^{3}  &  =q_{1}^{\left(
3\right)  }E\left(  2\right)  ,\\
\left(  \Sigma_{23}\left(  2\right)  \right)  ^{3}  &  =q_{2}^{\left(
3\right)  }E\left(  2\right)  ,\\
\left(  \Sigma_{32}\left(  2\right)  \right)  ^{3}  &  =q_{2}^{\left(
3\right)  }E\left(  2\right)  ,\\
\Sigma_{13}\left(  2\right)  \Sigma_{13}\left(  2\right)  E\left(  2\right)
&  =\Sigma_{31}\left(  2\right)  \Sigma_{31}\left(  2\right)  E\left(
2\right)  ,
\end{align}
where $q_{1}^{\left[  3\right]  }=\sigma_{1}\sigma_{2}\sigma_{1}=\sigma
_{2}\sigma_{1}\sigma_{2}$ and $q_{2}^{\left[  3\right]  }=\sigma_{2}\sigma
_{3}\sigma_{2}=\sigma_{3}\sigma_{2}\sigma_{3}$. The first two relations give
the braid relations (\ref{bs1})--(\ref{bs2}), while the last relation
corresponds to far commutativity (\ref{bs3}).
\end{example}

\section{\textsc{Generated }$k$\textsc{-ary matrix group corresponding the
higher braid group}}

The above construction of the ternary matrix group $\mathcal{M}_{3}%
^{gen\text{-}\Sigma}$ corresponding to the braid group $B_{n}$ can be
naturally extended to the $k$-ary case, which will allow us to
\textquotedblleft go in the opposite way\textquotedblright\ and build so
called higher\ degree analogs of $B_{n}$ (in our sense: the number of factors
in braid relations more than $3$). We denote such a
braid-like group with $n$ generators by $\mathcal{B}_{n}\left[  k\right]  $,
where $k$ is the number of generator multipliers in the braid relations (as in
the regularity relations (\ref{gg1})--(\ref{gg3})). Simultaneously $k$ is the
-arity of the matrices (\ref{m2}), we therefore call $\mathcal{B}_{n}\left[
k\right]  $ a higher $k$\textit{-degree analog} of the braid group $B_{n}$. In
this notation the Artin braid group $B_{n}$ is $\mathcal{B}_{n}\left[
3\right]  $. Now we build $\mathcal{B}_{n}\left[  k\right]  $ for any degree
$k$ exploiting the \textquotedblleft reverse\textquotedblright\ procedure, as
for $k=3$ and $B_{n}$ in \textsc{Section \ref{sec-tern}}. For that we need a
$k$-ary generalization of the matrices over $B_{n}$, which in the ternary case
are the anti-diagonal matrices $M\left(  2\right)  $ (\ref{m2}), and the
generator matrices $\Sigma_{ij}\left(  2\right)  $ (\ref{ss}). Then, using the
$k$-ary analog of multiplication (\ref{mb2})--(\ref{mb3}) we will obtain the
higher\ degree (than (\ref{s1})) braid relations which generate the so called
\textit{higher} $k$\textit{-degree braid group}. In distinction to the higher
degree regular semigroup construction from \textsc{Section \ref{sec-smgk}},
where the $k$-ary matrices form a semigroup for the Abelian group $G_{free}$,
using the generator matrices, we construct a $k$-ary matrix semigroup
(presented by generators and relations) for any (even non-commutative) matrix
entries. In this way the polyadic-binary correspondence will connect $k$-ary
matrix groups of finite order with higher binary braid groups (cf. idempotent
$k$-ary matrices and higher regular semigroups (\ref{ck})).

Let us consider a free binary group $\mathcal{B}_{free}$ and construct over it
a $k$-ary matrix group along the lines of \cite{nik84}, similarly to the
ternary matrix group $\mathcal{M}_{3}$ in (\ref{m2})--(\ref{mb3}).

\begin{definition}
A set $\mathrm{M}\left(  k-1\right)  =\left\{  M\left(  k-1\right)  \right\}
$ of $k$-ary $\left(  k-1\right)  \times\left(  k-1\right)  $ matrices%
\begin{align}
M\left(  k-1\right)   &  =M^{\left(  \left(  k-1\right)  \times\left(
k-1\right)  \right)  }\left(  \mathsf{b}^{\left(  1\right)  },\mathsf{b}%
^{\left(  2\right)  },\ldots,\mathsf{b}^{\left(  k-1\right)  }\right)
=\left(
\begin{array}
[c]{ccccc}%
0 & \mathsf{b}^{\left(  1\right)  } & 0 & \ldots & 0\\
0 & 0 & \mathsf{b}^{\left(  1\right)  } & \ldots & 0\\
0 & 0 & \vdots & \ddots & \vdots\\
\vdots & \vdots & 0 & \ddots & \mathsf{b}^{\left(  k-2\right)  }\\
\mathsf{b}^{\left(  k-1\right)  } & 0 & 0 & \ldots & 0
\end{array}
\right)  ,\label{mb}\\
\mathsf{b}^{\left(  j\right)  }  &  \in\mathcal{B}_{free},\ \ \ j=1,\ldots
,k-1,\nonumber
\end{align}
form a $k$-ary matrix semigroup $\mathcal{M}_{k}=\left\{  \mathrm{M}\left(
k-1\right)  \mid\mu_{k}\right\}  $, where $\mu_{k}$ is the $k$-ary
multiplication%
\begin{align}
\mu_{k}\left[  M_{1}\left(  k-1\right)  ,M_{2}\left(  k-1\right)
,\ldots,M_{k}\left(  k-1\right)  \right]   &  =M_{1}\left(  k-1\right)
M_{2}\left(  k-1\right)  \ldots M_{k}\left(  k-1\right)  =M\left(  k-1\right)
,\label{mkk}\\
\mathsf{b}_{1}^{\left(  1\right)  }\mathsf{b}_{2}^{\left(  2\right)  }%
,\ldots\mathsf{b}_{k-1}^{\left(  k-1\right)  }\mathsf{b}_{k}^{\left(
1\right)  }  &  =\mathsf{b}^{\left(  1\right)  },\label{b1}\\
\mathsf{b}_{1}^{\left(  2\right)  }\mathsf{b}_{2}^{\left(  3\right)  }%
,\ldots\mathsf{b}_{k-1}^{\left(  1\right)  }\mathsf{b}_{k}^{\left(  2\right)
}  &  =\mathsf{b}^{\left(  2\right)  },\label{b2}\\
&  \vdots\nonumber\\
\mathsf{b}_{1}^{\left(  k-1\right)  }\mathsf{b}_{2}^{\left(  1\right)
},\ldots\mathsf{b}_{k-1}^{\left(  k-2\right)  }\mathsf{b}_{k}^{\left(
k-1\right)  }  &  =\mathsf{b}^{\left(  k-1\right)  }, \label{b3}%
\end{align}
where the r.h.s. of (\ref{mkk}) is the ordinary matrix multiplication of
$k$-ary matrices (\ref{mb}) $M_{i}\left(  k-1\right)  =M^{\left(  \left(
k-1\right)  \times\left(  k-1\right)  \right)  }\left(  \mathsf{b}%
_{i}^{\left(  1\right)  },\mathsf{b}_{i}^{\left(  2\right)  },\ldots
,\mathsf{b}_{i}^{\left(  k-1\right)  }\right)  $, $i=1,\ldots,k$.
\end{definition}

\begin{proposition}
$\mathcal{M}_{k}$ is a $k$-ary matrix group.
\end{proposition}

\begin{proof}
Because $\mathcal{B}_{free}$ is a (binary) group with the identity
$\mathbf{e}\in\mathcal{B}_{free}$, each element of the $k$-ary matrix
semigroup $M\left(  k-1\right)  \in\mathcal{M}_{k}$ is invertible (in the
$k$-ary sense) and has a \textit{querelement} $\bar{M}\left(  k-1\right)  $
(see \cite{dor3}) defined by (cf. (\ref{mee}))%
\begin{equation}
\mu_{k}\left[  \overset{k-1}{\overbrace{M\left(  k-1\right)  ,\ldots,M\left(
k-1\right)  }},\bar{M}\left(  k-1\right)  \right]  =\ldots=M\left(
k-1\right)  , \label{mmk}%
\end{equation}
where $\bar{M}\left(  k-1\right)  $ can be on any place, and so we have $k$
conditions (cf. (\ref{m1}) for $k=3$).
\end{proof}

The $k$-ary matrix group has the polyadic identity%
\begin{equation}
E\left(  k-1\right)  =E^{\left(  \left(  k-1\right)  \times\left(  k-1\right)
\right)  }=\left(
\begin{array}
[c]{ccccc}%
0 & \mathbf{e} & 0 & \ldots & 0\\
0 & 0 & \mathbf{e} & \ldots & 0\\
0 & 0 & \vdots & \ddots & \vdots\\
\vdots & \vdots & 0 & \ddots & \mathbf{e}\\
\mathbf{e} & 0 & 0 & \ldots & 0
\end{array}
\right)  ,\ \ \ \mathbf{e}\in\mathcal{B}_{free}, \label{ek}%
\end{equation}
satisfying%
\begin{equation}
\mu_{k}\left[  M\left(  k-1\right)  ,E\left(  k-1\right)  ,\ldots,E\left(
k-1\right)  \right]  =\ldots=M\left(  k-1\right)  ,
\end{equation}
where $M\left(  k-1\right)  $ can be on any place, and so we have $k$
conditions (cf. (\ref{mee})).

\begin{definition}
An element of a $k$-ary group $M\left(  k-1\right)  \in\mathcal{M}_{k}$ has
the \textit{polyadic order }$\ell$, if%
\begin{equation}
\left(  M\left(  k-1\right)  \right)  ^{\left\langle \ell\right\rangle _{k}%
}\equiv\left(  M\left(  k-1\right)  \right)  ^{\ell\left(  k-1\right)
+1}=E\left(  k-1\right)  ,
\end{equation}
where $E\left(  k-1\right)  \in\mathcal{M}_{k}$ is the polyadic identity
(\ref{ek}), for $k=3$ see (\ref{e2}).
\end{definition}

\begin{definition}
An element ($\left(  k-1\right)  \times\left(  k-1\right)  $-matrix over
$\mathcal{B}_{free}$) $M\left(  k-1\right)  \in\mathcal{M}^{\left\{
k\right\}  }$ is of \textit{finite }$q$\textit{-polyadic order}, if there
exists a finite $\ell$ such that%
\begin{equation}
\left(  M\left(  k-1\right)  \right)  ^{\left\langle \ell\right\rangle _{k}%
}\equiv\left(  M\left(  k-1\right)  \right)  ^{\ell\left(  k-1\right)
+1}=\mathsf{q}E\left(  k-1\right)  ,\ \ \ \mathsf{q}\in\mathcal{B}_{free}.
\label{mq}%
\end{equation}

\end{definition}

Let us assume that the binary group $\mathcal{B}_{free}$ is presented by
generators and relations (cf. the Artin braid group (\ref{s1})--(\ref{s2})),
i.e. it is generated by $n-1$ generators $\mathbf{\sigma}_{i}$, $i=1,\ldots
,n-1$. An element of $\mathcal{B}_{n}^{gen\text{-}\mathbf{\sigma}}%
\equiv\mathcal{B}_{free}\left(  \mathbf{e},\mathbf{\sigma}_{i}\right)  $ is
the word of the form (\ref{w}). To find the relations between $\mathbf{\sigma
}_{i}$ we construct the corresponding $k$-ary matrix generators analogous to
the ternary ones (\ref{ss}). Then using a $k$-ary version of the relations
(\ref{ms1})--(\ref{ms3}) for the matrix generators, as the finite order
conditions (\ref{mq}), we will obtain the corresponding higher degree braid
relations for the binary generators $\mathbf{\sigma}_{i}$, and can therefore
present a higher degree braid group $\mathcal{B}_{n}\left[  k\right]  $ in the
form of generators and relations.

Using $n-1$ generators $\mathbf{\sigma}_{i}$ of $\mathcal{B}_{n}%
^{gen\text{-}\mathbf{\sigma}}$ we build $\left(  n-1\right)  ^{k}$ polyadic
(or $k$-ary) $\left(  k-1\right)  \times\left(  k-1\right)  $-matrix
generators having $k-1$ indices $i_{1},\ldots,i_{k-1}=1,\ldots,n-1$, as
follows%
\begin{equation}
\Sigma_{i_{1},\ldots,i_{k-1}}\left(  k-1\right)  \equiv\Sigma_{i_{1}%
,\ldots,i_{k-1}}^{\left(  \left(  k-1\right)  \times\left(  k-1\right)
\right)  }\left(  \mathbf{\sigma}_{i_{1}},\ldots,\mathbf{\sigma}_{i_{k-1}%
}\right)  =\left(
\begin{array}
[c]{ccccc}%
0 & \mathbf{\sigma}_{i_{1}} & 0 & \ldots & 0\\
0 & 0 & \mathbf{\sigma}_{i_{2}} & \ldots & 0\\
0 & 0 & \vdots & \ddots & \vdots\\
\vdots & \vdots & 0 & \ddots & \mathbf{\sigma}_{i_{k-2}}\\
\mathbf{\sigma}_{i_{k-1}} & 0 & 0 & \ldots & 0
\end{array}
\right)  .\label{sk}%
\end{equation}

For the matrix generator $\Sigma_{i_{1},\ldots,i_{k-1}}\left(  k-1\right)  $
(\ref{sk}) its querelement $\bar{\Sigma}_{i_{1},\ldots,i_{k-1}}\left(
k-1\right)  $ is defined by (\ref{mmk}).

We now build a $k$-ary matrix analog of the braid relations (\ref{s1}),
(\ref{ms1}) and of far commutativity (\ref{s2}), (\ref{ms3}). Using (\ref{sk})
we obtain $\left(  k-1\right)  $ conditions that the matrix generators are of
finite polyadic order (analog of (\ref{ms1}))%

\begin{align}
&  \mu_{k}\left[  \Sigma_{i,i+1,\ldots,i+k-2}\left(  k-1\right)
,\Sigma_{i,i+1,\ldots,i+k-2}\left(  k-1\right)  ,\ldots,\Sigma_{i,i+1,\ldots
,i+k-2}\left(  k-1\right)  \right] \label{me1}\\
&  =\mu_{k}\left[  \Sigma_{i+1,i+2,\ldots,i+k-2,i}\left(  k-1\right)
,\Sigma_{i+1,i+2,\ldots,i+k-2,i}\left(  k-1\right)  ,\ldots,\Sigma
_{i+1,i+2,\ldots,i+k-2,i}\left(  k-1\right)  \right] \\
&  \vdots\nonumber\\
&  \mu_{k}\left[  \Sigma_{i+k-2,i,i+1,\ldots,i+k-3}\left(  k-1\right)
,\Sigma_{i+k-2,i,i+1,\ldots,i+k-3}\left(  k-1\right)  ,\ldots,\Sigma
_{i+k-2,i,i+1,\ldots,i+k-3}\left(  k-1\right)  \right] \\
&  =q_{i}^{\left[  k\right]  }E\left(  k-1\right)  ,\ \ 1\leq i\leq n-k+1,
\label{me2}%
\end{align}
where $E\left(  k-1\right)  $ are polyadic identities (\ref{ek}) and
$q_{i}^{\left[  k\right]  }\in\mathcal{B}_{n}^{gen\text{-}\mathbf{\sigma}}$.

We propose a $k$-ary version of the far commutativity relation (\ref{ms3}) in
the following form%
\begin{align}
&  \mu_{k}\left[  \overset{k-1}{\overbrace{\Sigma_{i_{1},\ldots,i_{k-1}%
}\left(  k-1\right)  ,\ldots,\Sigma_{i_{1},\ldots,i_{k-1}}\left(  k-1\right)
,}}E\left(  k-1\right)  \right]  =\ldots\\
&  =\mu_{k}\left[  \overset{k-1}{\overbrace{\Sigma_{\tau\left(  i_{1}\right)
,\tau\left(  i_{2}\right)  ,\ldots,\tau\left(  i_{k-1}\right)  }\left(
k-1\right)  ,\ldots,\Sigma_{\tau\left(  i_{1}\right)  ,\tau\left(
i_{2}\right)  ,\ldots,\tau\left(  i_{k-1}\right)  }\left(  k-1\right)  ,}%
}E\left(  k-1\right)  \right]  ,\\
\text{if all }\left\vert i_{p}-i_{s}\right\vert  &  \geq
k-1,\ \ \ \ p,s=1,\ldots,k-1,
\end{align}
where $\tau$ is an element the permutation symmetry group $\tau\in S_{k-1}$.

In matrix form we can define

\begin{definition}
A $k$\textit{-ary (generated) matrix group} $\mathcal{M}_{k}^{gen\text{-}%
\Sigma}$ is presented by the $\left(  k-1\right)  \times\left(  k-1\right)  $
matrix generators $\Sigma_{i_{1},\ldots,i_{k-1}}\left(  k-1\right)  $
(\ref{sk}) and the relations (we use (\ref{mkk}))%
\begin{align}
&  \left(  \Sigma_{i,i+1,\ldots,i+k-2}\left(  k-1\right)  \right)
^{k}\label{sk1}\\
&  =\left(  \Sigma_{i+1,i+2,\ldots,i+k-2,i}\left(  k-1\right)  \right)  ^{k}\\
&  \vdots\nonumber\\
&  \left(  \Sigma_{i+k-2,i,i+1,\ldots,i+k-3}\left(  k-1\right)  \right)
^{k}\label{sk2}\\
&  =q_{i}^{\left[  k\right]  }E\left(  k-1\right)  ,\ \ 1\leq i\leq
n-k+1,\nonumber
\end{align}
and%
\begin{align}
\left(  \Sigma_{i_{1},i_{2},\ldots,i_{k-1}}^{\left(  k-1\right)  }\right)
^{k-1}E\left(  k-1\right)   &  =\left(  \Sigma_{\tau\left(  i_{1}\right)
,\tau\left(  i_{2}\right)  ,\ldots,\tau\left(  i_{k-1}\right)  }^{\left(
k-1\right)  }\right)  ^{k-1}E\left(  k-1\right)  ,\label{se}\\
\text{if all }\left\vert i_{p}-i_{s}\right\vert  &  \geq
k-1,\ \ \ \ p,s=1,\ldots,k-1,\nonumber
\end{align}
where $\tau\in S_{k-1}$ and $q_{i}^{\left[  k\right]  }\in\mathcal{B}%
_{n}^{gen\text{-}\mathbf{\sigma}}$.
\end{definition}

Each element of $\mathcal{M}_{k}^{gen\text{-}\Sigma}$ is a $k$-ary matrix word
(analogous to the binary word (\ref{w})) being the $k$-ary product of the
polyadic powers (\ref{ml}) of the matrix generators $\Sigma_{i_{1}%
,\ldots,i_{k-1}}\left(  k-1\right)  $ and their querelements $\bar{\Sigma
}_{i_{1},\ldots,i_{k-1}}\left(  k-1\right)  $ as in (\ref{ww}).

Similarly to the ternary case $k=3$ (\textsc{Section \ref{sec-tern}})\textsc{
}we now develop the $k$-ary \textquotedblleft reverse\textquotedblright%
\ procedure and build from $\mathcal{B}_{n}^{gen\text{-}\mathbf{\sigma}}$ the
higher $k$-degree braid group $\mathcal{B}_{n}\left[  k\right]  $ using
(\ref{sk}). Because the presentation of $\mathcal{M}_{k}^{gen\text{-}\Sigma}$
by generators and relations has already been given in (\ref{sk1})--(\ref{se}),
we need to expand them into components and postulate that these new relations
between the (binary) generators $\mathbf{\sigma}_{i}$ present a new higher
degree analog of the braid group. This gives

\begin{definition}
A \textit{higher }$k$\textit{-degree braid (binary) group} $\mathcal{B}%
_{n}\left[  k\right]  $ is presented by $\left(  n-1\right)  $ generators
$\mathbf{\sigma}_{i}\equiv\mathbf{\sigma}_{i}^{\left[  k\right]  }$ (and the
identity $\mathbf{e}$) satisfying the following relations

$\bullet$ $\left(  k-1\right)  $ higher braid relations%
\begin{align}
&  \overset{k}{\overbrace{\mathbf{\sigma}_{i}\mathbf{\sigma}_{i+1}%
\ldots\mathbf{\sigma}_{i+k-3}\mathbf{\sigma}_{i+k-2}\mathbf{\sigma}_{i}}%
}\label{ss1}\\
&  =\mathbf{\sigma}_{i+1}\mathbf{\sigma}_{i+2}\ldots\mathbf{\sigma}%
_{i+k-2}\mathbf{\sigma}_{i}\mathbf{\sigma}_{i+1}\label{ss2}\\
&  \vdots\nonumber\\
&  =\mathbf{\sigma}_{i+k-2}\mathbf{\sigma}_{i}\mathbf{\sigma}_{i+1}%
\mathbf{\sigma}_{i+2}\ldots\mathbf{\sigma}_{i}\mathbf{\sigma}_{i+1}%
\mathbf{\sigma}_{i+k-2}\equiv q_{i}^{\left[  k\right]  },\ \ \ \ q_{i}%
^{\left[  k\right]  }\in\mathcal{B}_{n}\left[  k\right]  ,\label{ss3}\\
\ \ i  &  \in\mathfrak{I}_{braid}=\left\{  1,\ldots,n-k+1\right\}  ,
\label{ibr}%
\end{align}

$\bullet$ $\left(  k-1\right)  $-ary far commutativity%
\begin{align}
&  \overset{k-1}{\overbrace{\mathbf{\sigma}_{i_{1}}\mathbf{\sigma}_{i_{2}%
}\ldots\mathbf{\sigma}_{i_{k-3}}\mathbf{\sigma}_{i_{k-2}}\mathbf{\sigma
}_{i_{k-1}}}}\label{kc1}\\
&  \vdots\nonumber\\
&  =\mathbf{\sigma}_{\tau\left(  i_{1}\right)  }\mathbf{\sigma}_{\tau\left(
i_{2}\right)  }\ldots\mathbf{\sigma}_{\tau\left(  i_{k-3}\right)
}\mathbf{\sigma}_{\tau\left(  i_{k-2}\right)  }\mathbf{\sigma}_{\tau\left(
i_{k-1}\right)  },\label{kc2}\\
\text{if all }\left\vert i_{p}-i_{s}\right\vert  &  \geq
k-1,\ \ \ \ p,s=1,\ldots,k-1,\label{ip}\\
\ \ \mathfrak{I}_{far}  &  =\left\{  n-k,\ldots,n-1\right\}  , \label{ifar}%
\end{align}
where $\tau$ is an element of the permutation symmetry group $\tau\in S_{k-1}$.
\end{definition}

A general element of the higher $k$-degree braid group $\mathcal{B}_{n}\left[
k\right]  $ is a word of the form%
\begin{equation}
\mathfrak{w}=\mathbf{\sigma}_{i_{1}}^{p_{1}}\ldots\mathbf{\sigma}_{i_{r}%
}^{p_{r}}\ldots\mathbf{\sigma}_{i_{m}}^{p_{m}},\ \ \ i_{m}=1,\ldots,n,
\label{wk}%
\end{equation}
where $p_{r}\in\mathbb{Z}$ are (positive or negative) powers of the generators
$\sigma_{i_{r}}$, $r=1,\ldots,m$ and $m\in\mathbb{N}$.

\begin{remark}
The ternary case $k=3$ coincides with the Artin braid group $\mathcal{B}%
_{n}^{\left[  3\right]  }=B_{n}$ (\ref{s1})--(\ref{s2}).
\end{remark}

\begin{remark}
The representation of the higher $k$-degree braid relations in $\mathcal{B}%
_{n}\left[  k\right]  $ in the tensor product of vector spaces (similarly to
$B_{n}$ and the Yang-Baxter equation \cite{tur88}) can be obtained using the
$n^{\prime}$-ary braid equation introduced in \cite{dup2018d}
(\textbf{Proposition 7.2} and next there).
\end{remark}

\begin{definition}
We say that the $k$-ary matrix group $\mathcal{M}_{k}^{gen\text{-}\Sigma}$
generated by the matrix generators $\Sigma_{i_{1},i_{2},\ldots,i_{k-1}}\left(
k-1\right)  $ satisfying the relations (\ref{sk1})--(\ref{se}) is in
\textit{polyadic-binary correspondence} with the higher $k$-degree braid group
$\mathcal{B}_{n}\left[  k\right]  $, which is denoted as (cf. (\ref{c3}))%
\begin{equation}
\mathcal{M}_{k}^{gen\text{-}\Sigma}\Bumpeq\mathcal{B}_{n}\left[  k\right]  .
\label{ckb}%
\end{equation}

\end{definition}

\begin{example}
Let $k=4$, then the $4$-ary matrix group $\mathcal{M}_{4}^{gen\text{-}\Sigma}$
is generated by the matrix generators $\Sigma_{i_{1},i_{2},i_{3}}\left(
3\right)  $ satisfying (\ref{sk1})--(\ref{se})

$\bullet$ $4$-ary relations of $q$-polyadic order (\ref{me})%
\begin{equation}
\left(  \Sigma_{i,i+1,i+2}\left(  3\right)  \right)  ^{4}=\left(
\Sigma_{i+2,i,i+1}\left(  3\right)  \right)  ^{4}=\left(  \Sigma
_{i+1,i+2,i}\left(  3\right)  \right)  ^{4}=q_{i}^{\left[  3\right]  }E\left(
3\right)  ,\ \ 1\leq i\leq n-3, \label{o4}%
\end{equation}
$\bullet$ far commutativity%
\begin{align}
&  \left(  \Sigma_{i_{1},i_{2},i_{3}}\left(  3\right)  \right)  ^{3}E\left(
3\right)  =\left(  \Sigma_{i_{3},i_{1},i_{2}}\left(  3\right)  \right)
^{3}E\left(  3\right)  =\left(  \Sigma_{i_{2},i_{3},i_{1}}\left(  3\right)
\right)  ^{3}E\left(  3\right) \nonumber\\
&  =\left(  \Sigma_{i_{1},i_{3},i_{2}}\left(  3\right)  \right)  ^{3}E\left(
3\right)  =\left(  \Sigma_{i_{3},i_{2},i_{1}}\left(  3\right)  \right)
^{3}E\left(  3\right)  =\left(  \Sigma_{i_{2},i_{1},i_{3}}\left(  3\right)
\right)  ^{3}E\left(  3\right)  ,\label{f4}\\
\left\vert i_{1}-i_{2}\right\vert  &  \geq3,\ \ \left\vert i_{1}%
-i_{3}\right\vert \geq3,\ \ \left\vert i_{2}-i_{3}\right\vert \geq3.\nonumber
\end{align}

Let $\mathbf{\sigma}_{i}\equiv\mathbf{\sigma}_{i}^{\left[  4\right]  }%
\in\mathcal{B}_{n}\left[  4\right]  $, $i=1,\ldots,n-1$, then we use the
$4$-ary $3\times3$ matrix presentation for the generators (cf.
\textit{Example} \ref{ex-smgk4})
\begin{equation}
\Sigma_{i_{1},i_{2},i_{3}}\left(  3\right)  \equiv\Sigma^{\left(
3\times3\right)  }\left(  \mathbf{\sigma}_{i_{1}},\mathbf{\sigma}_{i_{2}%
},\mathbf{\sigma}_{i_{3}}\right)  =\left(
\begin{array}
[c]{ccc}%
0 & \mathbf{\sigma}_{i_{1}} & 0\\
0 & 0 & \mathbf{\sigma}_{i_{2}}\\
\mathbf{\sigma}_{i_{3}} & 0 & 0
\end{array}
\right)  ,\ \ \ i_{1},i_{2},i_{3}=1,\ldots,n-1.
\end{equation}
The querelement $\bar{\Sigma}_{i_{1},i_{2},i_{3}}\left(  3\right)  $
satisfying%
\begin{equation}
\left(  \Sigma_{i_{1},i_{2},i_{3}}\left(  3\right)  \right)  ^{3}\bar{\Sigma
}_{i_{1},i_{2},i_{3}}\left(  3\right)  =\Sigma_{i_{1},i_{2},i_{3}}\left(
3\right)  ,
\end{equation}
has the form%
\begin{equation}
\bar{\Sigma}_{i_{1},i_{2},i_{3}}\left(  3\right)  =\left(
\begin{array}
[c]{ccc}%
0 & \mathbf{\sigma}_{i_{3}}^{-1}\mathbf{\sigma}_{i_{2}}^{-1} & 0\\
0 & 0 & \mathbf{\sigma}_{i_{1}}^{-1}\mathbf{\sigma}_{i_{3}}^{-1}\\
\mathbf{\sigma}_{i_{2}}^{-1}\mathbf{\sigma}_{i_{1}}^{-1} & 0 & 0
\end{array}
\right)  .
\end{equation}

Expanding (\ref{o4})--(\ref{f4}) in components, we obtain the relations for
the higher $4$-degree braid group $\mathcal{B}_{n}\left[  4\right]  $ as follows

$\bullet$ higher $4$-degree braid relations%
\begin{equation}
\mathbf{\sigma}_{i}\mathbf{\sigma}_{i+1}\mathbf{\sigma}_{i+2}\mathbf{\sigma
}_{i}=\mathbf{\sigma}_{i+1}\mathbf{\sigma}_{i+2}\mathbf{\sigma}_{i}%
\mathbf{\sigma}_{i+1}=\mathbf{\sigma}_{i+2}\mathbf{\sigma}_{i}\mathbf{\sigma
}_{i+1}\mathbf{\sigma}_{i+2}\equiv q_{i}^{\left[  4\right]  },\ \ 1\leq i\leq
n-3, \label{s4}%
\end{equation}

$\bullet$ ternary far (total) commutativity%
\begin{align}
\mathbf{\sigma}_{i_{1}}\mathbf{\sigma}_{i_{2}}\mathbf{\sigma}_{i_{3}} &
=\mathbf{\sigma}_{i_{2}}\mathbf{\sigma}_{i_{3}}\mathbf{\sigma}_{i_{1}%
}=\mathbf{\sigma}_{i_{3}}\mathbf{\sigma}_{i_{1}}\mathbf{\sigma}_{i_{2}%
}=\mathbf{\sigma}_{i_{1}}\mathbf{\sigma}_{i_{3}}\mathbf{\sigma}_{i_{2}%
}=\mathbf{\sigma}_{i_{2}}\mathbf{\sigma}_{i_{1}}\mathbf{\sigma}_{i_{3}%
}=\mathbf{\sigma}_{i_{3}}\mathbf{\sigma}_{i_{2}}\mathbf{\sigma}_{i_{1}%
},\label{sss}\\
\left\vert i_{1}-i_{2}\right\vert  &  \geq3,\ \ \left\vert i_{1}%
-i_{3}\right\vert \geq3,\ \ \left\vert i_{2}-i_{3}\right\vert \geq
3.\label{i12}%
\end{align}

\end{example}

In the higher $4$-degree braid group the minimum number of generators is $4$,
which follows from (\ref{s4}). In this case we have a braid relation for $i=1$
only and no far commutativity relations, because of (\ref{i12}). Then

\begin{example}
The higher $4$-degree braid group $\mathcal{B}_{4}\left[  4\right]  $ is
generated by $3$ generators $\mathbf{\sigma}_{1}$, $\mathbf{\sigma}_{2}$,
$\mathbf{\sigma}_{3}$, which satisfy only the braid relation%
\begin{equation}
\mathbf{\sigma}_{1}\mathbf{\sigma}_{2}\mathbf{\sigma}_{3}\mathbf{\sigma}%
_{1}=\mathbf{\sigma}_{2}\mathbf{\sigma}_{3}\mathbf{\sigma}_{1}\mathbf{\sigma
}_{2}=\mathbf{\sigma}_{3}\mathbf{\sigma}_{1}\mathbf{\sigma}_{2}\mathbf{\sigma
}_{3}.\label{b44}%
\end{equation}

\end{example}

If $n\leq7$, then there will no far commutativity relations at all, which
follows from (\ref{i12}), and so the first higher $4$-degree braid group
containing far commutativity should have $n=8$ elements.

\begin{example}
\label{ex-b48}The higher $4$-degree braid group $\mathcal{B}_{8}\left[
4\right]  $ is generated by $7$ generators $\mathbf{\sigma}_{1},\ldots
,\mathbf{\sigma}_{7}$, which satisfy the braid relations with $i=1,\ldots,5$%
\begin{align}
\mathbf{\sigma}_{1}\mathbf{\sigma}_{2}\mathbf{\sigma}_{3}\mathbf{\sigma}_{1}
&  =\mathbf{\sigma}_{2}\mathbf{\sigma}_{3}\mathbf{\sigma}_{1}\mathbf{\sigma
}_{2}=\mathbf{\sigma}_{3}\mathbf{\sigma}_{1}\mathbf{\sigma}_{2}\mathbf{\sigma
}_{3},\\
\mathbf{\sigma}_{2}\mathbf{\sigma}_{3}\mathbf{\sigma}_{4}\mathbf{\sigma}_{2}
&  =\mathbf{\sigma}_{3}\mathbf{\sigma}_{4}\mathbf{\sigma}_{2}\mathbf{\sigma
}_{3}=\mathbf{\sigma}_{4}\mathbf{\sigma}_{2}\mathbf{\sigma}_{3}\mathbf{\sigma
}_{4},\\
\mathbf{\sigma}_{3}\mathbf{\sigma}_{4}\mathbf{\sigma}_{5}\mathbf{\sigma}_{3}
&  =\mathbf{\sigma}_{4}\mathbf{\sigma}_{5}\mathbf{\sigma}_{3}\mathbf{\sigma
}_{4}=\mathbf{\sigma}_{5}\mathbf{\sigma}_{3}\mathbf{\sigma}_{4}\mathbf{\sigma
}_{5},\\
\mathbf{\sigma}_{4}\mathbf{\sigma}_{5}\mathbf{\sigma}_{6}\mathbf{\sigma}_{4}
&  =\mathbf{\sigma}_{5}\mathbf{\sigma}_{6}\mathbf{\sigma}_{4}\mathbf{\sigma
}_{5}=\mathbf{\sigma}_{6}\mathbf{\sigma}_{4}\mathbf{\sigma}_{5}\mathbf{\sigma
}_{6},\\
\mathbf{\sigma}_{5}\mathbf{\sigma}_{6}\mathbf{\sigma}_{7}\mathbf{\sigma}_{5}
&  =\mathbf{\sigma}_{6}\mathbf{\sigma}_{7}\mathbf{\sigma}_{5}\mathbf{\sigma
}_{6}=\mathbf{\sigma}_{7}\mathbf{\sigma}_{5}\mathbf{\sigma}_{6}\mathbf{\sigma
}_{7},
\end{align}
together with the ternary far commutativity relation%
\begin{equation}
\mathbf{\sigma}_{1}\mathbf{\sigma}_{4}\mathbf{\sigma}_{7}=\mathbf{\sigma}%
_{4}\mathbf{\sigma}_{7}\mathbf{\sigma}_{1}=\mathbf{\sigma}_{7}\mathbf{\sigma
}_{1}\mathbf{\sigma}_{4}=\mathbf{\sigma}_{1}\mathbf{\sigma}_{7}\mathbf{\sigma
}_{4}=\mathbf{\sigma}_{4}\mathbf{\sigma}_{1}\mathbf{\sigma}_{7}=\mathbf{\sigma
}_{7}\mathbf{\sigma}_{4}\mathbf{\sigma}_{1}.
\end{equation}

\end{example}

\begin{remark}
In polyadic group theory there are several possible modifications of the
commutativity property, but nevertheless we assume here the \textit{total
commutativity} relations in the $k$-ary matrix generators and the
corresponding far commutativity relations in the higher degree braid groups.
\end{remark}

If $\mathcal{B}_{n}\left[  k\right]  \rightarrow\mathbb{Z}$ is the
abelianization defined by $\mathbf{\sigma}_{i}^{\pm}\rightarrow\pm1$, then
$\mathbf{\sigma}_{i}^{p}=\mathbf{e}$, if and only if $p=0$, and
$\mathbf{\sigma}_{i}$ are of infinite order. Moreover, we can prove (as in the
ordinary case $k=3$ \cite{dye80})

\begin{theorem}
The higher $k$-degree braid group $\mathcal{B}_{n}\left[  k\right]  $ is torsion-free.
\end{theorem}

Recall (see, e.g. \cite{kas/tur2008}) that there exists a surjective
homomorphism of the braid group onto the finite symmetry group $B_{n}%
\rightarrow S_{n}$ by $\sigma_{i}\rightarrow s_{i}=\left(  i,i+1\right)  \in
S_{n}$. The generators $s_{i}$ satisfy (\ref{s1})--(\ref{s2}) together with
the finite order demand%
\begin{align}
s_{i}s_{i+1}s_{i} &  =s_{i+1}s_{i}s_{i+1},\ \ \ 1\leq i\leq n-2,\label{sc1}\\
s_{i}s_{j} &  =s_{j}s_{i},\ \ \ \ \left\vert i-j\right\vert \geq
2,\label{sc2}\\
s_{i}^{2} &  =e,\ \ \ \ \ i=1,\ldots,n-1,\label{sc3}%
\end{align}
which is called the \textit{Coxeter presentation} of the symmetry group
$S_{n}$. Indeed, multiplying both sides of (\ref{sc1}) from the right
successively by $s_{i+1}$, $s_{i}$, and $s_{i+1}$, using (\ref{sc3}), we
obtain $\left(  s_{i}s_{i+1}\right)  ^{3}=1$, and (\ref{sc2}) on $s_{i}$ and
$s_{j}$, we get $\left(  s_{i}s_{j}\right)  ^{2}=1$. Therefore, a Coxeter
group \cite{bri/sai} corresponding (\ref{sc1})--(\ref{sc3}) is presented by
the same generators $s_{i}$ and the relations%
\begin{align}
\left(  s_{i}s_{i+1}\right)  ^{3} &  =1,\ \ \ 1\leq i\leq n-2,\label{scc1}\\
\left(  s_{i}s_{j}\right)  ^{2} &  =1,\ \ \ \ \left\vert i-j\right\vert
\geq2,\label{scc2}\\
s_{i}^{2} &  =e,\ \ \ \ \ i=1,\ldots,n-1.\label{scc3}%
\end{align}

A general Coxeter group $W_{n}=W_{n}\left(  e,r_{i}\right)  $ is presented by
$n$ generators $r_{i}$ and the relations \cite{bjo/bre}%
\begin{equation}
\left(  r_{i}r_{j}\right)  ^{m_{ij}}=e,\ \ \ \ m_{ij}=\left\{
\begin{array}
[c]{c}%
1,\ i=j,\\
\geq2,\ \ i\neq j.
\end{array}
\right.  \label{rr}%
\end{equation}

By analogy with (\ref{sc1})--(\ref{sc3}), we make the following

\begin{definition}
A higher analog of $S_{n}$, the $k$\textit{-degree symmetry group}
$\mathcal{S}_{n}\left[  k\right]  =\mathcal{S}_{n}^{\left[  k\right]  }\left(
\mathbf{e},\mathbf{s}_{i}\right)  $, is presented by generators $\mathbf{s}%
_{i}$, $i=1,\ldots,n-1$ satisfying (\ref{ss1})--(\ref{kc2}) together with the
additional condition of finite $\left(  k-1\right)  $-order $\mathbf{s}%
_{i}^{\left(  k-1\right)  }=\mathbf{e}$, $i=1,\ldots,n$.
\end{definition}

\begin{example}
The lowest higher degree case is $\mathcal{S}_{4}\left[  4\right]  $ which is
presented by three generators $\mathbf{s}_{1}$, $\mathbf{s}_{2}$,
$\mathbf{s}_{3}$ satisfying (see (\ref{b44}))%
\begin{align}
\mathbf{s}_{1}\mathbf{s}_{2}\mathbf{s}_{3}\mathbf{s}_{1} &  =\mathbf{s}%
_{2}\mathbf{s}_{3}\mathbf{s}_{1}\mathbf{s}_{2}=\mathbf{s}_{3}\mathbf{s}%
_{1}\mathbf{s}_{2}\mathbf{s}_{3},\\
\mathbf{s}_{1}^{3} &  =\mathbf{s}_{2}^{3}=\mathbf{s}_{3}^{3}=\mathbf{e}.
\end{align}

\end{example}

In a similar way we define a higher degree analog of the Coxeter group
(\ref{rr}).

\begin{definition}
A higher $k$\textit{-degree Coxeter group} $\mathcal{W}_{n}\left[  k\right]
=\mathcal{W}_{n}^{\left[  k\right]  }\left(  \mathbf{e},\mathbf{r}_{i}\right)
$ is presented by $n$ generators $\mathbf{r}_{i}$ obeying the relations%
\begin{align}
&  \left(  \mathbf{r}_{i_{1}},\mathbf{r}_{i_{2}},\ldots,\mathbf{r}_{i_{k-1}%
}\right)  ^{\mathbf{m}_{_{i_{1}i_{2}},\ldots,_{i_{k-1}}}}=\mathbf{e,}%
\label{re1}\\
&  \mathbf{m}_{_{i_{1}i_{2}},\ldots,_{i_{k-1}}}=\left\{
\begin{array}
[c]{c}%
1,\ \ \ i_{1}=i_{2}=\ldots=i_{k-1},\\
\geq k-1,\ \left\vert i_{p}-i_{s}\right\vert \geq k-1,\ \ \ \ p,s=1,\ldots
,k-1.
\end{array}
\right.  \label{re2}%
\end{align}

\end{definition}

It follows from (\ref{re2}) that all generators are of $\left(  k-1\right)  $
order $\mathbf{r}_{i}^{k-1}=\mathbf{e}$. A \textit{higher }$k$\textit{-degree
Coxeter matrix} is a hypermatrix $M_{n,Cox}^{\left[  k-1\right]  }\left(
\overset{k-1}{\overbrace{n\times n\times\ldots\times n}}\right)  $ having $1$
on the main diagonal and other entries $\mathbf{m}_{_{i_{1}i_{2}}%
,\ldots,_{i_{k-1}}}$.

\begin{example}
In the lowest higher degree case $k=4$ and all $\mathbf{m}_{_{i_{1}i_{2}%
},\ldots,_{i_{k-1}}}=3$ we have (instead of commutativity in the ordinary case
$k=3$)%
\begin{align}
\left(  \mathbf{r}_{i}\mathbf{r}_{j}\right)  ^{2} &  =\mathbf{r}_{j}%
^{2}\mathbf{r}_{i}^{2},\\
\mathbf{r}_{i}\mathbf{r}_{j}\mathbf{r}_{i} &  =\mathbf{r}_{j}^{2}%
\mathbf{r}_{i}^{2}\mathbf{r}_{j}^{2}.
\end{align}

\end{example}

\begin{example}
A higher $4$-degree analog of (\ref{scc1})--(\ref{scc3}) is given by%
\begin{align}
\left(  \mathbf{r}_{i}\mathbf{r}_{i+1}\mathbf{r}_{i+2}\right)  ^{4}  &
=1,\ \ \ 1\leq i\leq n-3,\label{r1}\\
\left(  \mathbf{r}_{i_{1}}\mathbf{r}_{i_{2}}\mathbf{r}_{i_{3}}\right)  ^{3}
&  =1,\ \ \left\vert i_{1}-i_{2}\right\vert \geq3,\ \ \left\vert i_{1}%
-i_{3}\right\vert \geq3,\ \ \left\vert i_{2}-i_{3}\right\vert \geq
3,\label{r2}\\
\mathbf{r}_{i}^{3}  &  =\mathbf{e},\ \ \ \ \ i=1,\ldots,n-1. \label{r3}%
\end{align}

It follows from (\ref{r2}), that%
\begin{equation}
\left(  \mathbf{r}_{i_{1}}\mathbf{r}_{i_{2}}\mathbf{r}_{i_{3}}\right)
^{2}=\mathbf{r}_{i_{3}}^{2}\mathbf{r}_{i_{2}}^{2}\mathbf{r}_{i_{1}}^{2},
\end{equation}
which cannot be reduced to total commutativity (\ref{sss}). From the first
relation (\ref{r1}) we obtain%
\begin{equation}
\mathbf{r}_{i}\mathbf{r}_{i+1}\mathbf{r}_{i+2}\mathbf{r}_{i}=\mathbf{r}%
_{i+2}^{2}\mathbf{r}_{i+1}^{2},
\end{equation}
which differs from the higher $4$-degree braid relations (\ref{s4}).
\end{example}

\begin{example}
In the simplest case the higher $4$-degree Coxeter group $\mathcal{W}%
_{4}\left[  4\right]  $ has $3$ generator $\mathbf{r}_{1}$, $\mathbf{r}_{2}$,
$\mathbf{r}_{3}$ satisfying%
\begin{equation}
\left(  \mathbf{r}_{1}\mathbf{r}_{2}\mathbf{r}_{3}\right)  ^{4}=\mathbf{r}%
_{1}^{3}=\mathbf{r}_{2}^{3}=\mathbf{r}_{3}^{3}=\mathbf{e}.
\end{equation}

\end{example}

\begin{example}
The minimal case, when the conditions (\ref{r2}) appear is $\mathcal{W}%
_{8}\left[  4\right]  $%
\begin{align}
\mathbf{r}_{1}\mathbf{r}_{2}\mathbf{r}_{3}\mathbf{r}_{1}  &  =\mathbf{r}%
_{3}^{2}\mathbf{r}_{2}^{2},\\
\mathbf{r}_{2}\mathbf{r}_{3}\mathbf{r}_{4}\mathbf{r}_{2}  &  =\mathbf{r}%
_{4}^{2}\mathbf{r}_{3}^{2},\\
\mathbf{r}_{3}\mathbf{r}_{4}\mathbf{r}_{5}\mathbf{r}_{3}  &  =\mathbf{r}%
_{5}^{2}\mathbf{r}_{4}^{2},\\
\mathbf{r}_{4}\mathbf{r}_{5}\mathbf{r}_{6}\mathbf{r}_{4}  &  =\mathbf{r}%
_{6}^{2}\mathbf{r}_{5}^{2},\\
\mathbf{r}_{5}\mathbf{r}_{6}\mathbf{r}_{7}\mathbf{r}_{5}  &  =\mathbf{r}%
_{7}^{2}\mathbf{r}_{6}^{2},
\end{align}
and an analog of commutativity%
\begin{equation}
\left(  \mathbf{r}_{1}\mathbf{r}_{4}\mathbf{r}_{7}\right)  ^{2}=\mathbf{r}%
_{7}^{2}\mathbf{r}_{4}^{2}\mathbf{r}_{1}^{2}.
\end{equation}

\end{example}

Thus, we arrive at

\begin{theorem}
The higher $k$-degree Coxeter group can present the $k$-degree symmetry group in
the lowest case only, if and only if $k=3$.
\end{theorem}

%\bigskip

As a further development, it would be interesting to consider the higher degree
(in our sense) groups constructed here from a geometric viewpoint (e.g.,
\cite{birman,kauffman}).

%\smallskip
\textbf{Acknowledgement}. The author is grateful to Mike Hewitt, Thomas Nordahl, Vladimir Tkach and Raimund Vogl for the numerous fruitful discussions and valuable support.

%\newpage
%\pagestyle{fancyref}
\pagestyle{emptyf}
%\mbox{}
%\vskip 0.5cm


\begin{thebibliography}{}

\bibitem[\protect\citeauthoryear{Artin}{\textcolor{blue}{\sc
  Artin}}{1947}]{art47}
{\textcolor{blue}{\sc Artin, E.}} (1947).
\newblock Theory of braids.
\newblock {\em Ann. Math.\/}~{\bf 48} (1), 101--126.

\bibitem[\protect\citeauthoryear{Birman}{\textcolor{blue}{\sc
  Birman}}{1976}]{birman}
{\textcolor{blue}{\sc Birman, J.~S.}} (1976).
\newblock {\em Braids, Links and Mapping Class Group}.
\newblock Princeton: Princeton University Press.

\bibitem[\protect\citeauthoryear{Bj\"{o}rner and Brenti}{\textcolor{blue}{\sc
  Bj\"{o}rner and Brenti}}{2005}]{bjo/bre}
{\textcolor{blue}{\sc Bj\"{o}rner, A. and F.~Brenti}} (2005).
\newblock {\em Combinatorics of {C}oxeter groups}.
\newblock New York: Springer.

\bibitem[\protect\citeauthoryear{Borowiec, Dudek, and
  Duplij}{\textcolor{blue}{\sc Borowiec et~al.}}{2006}]{bor/dud/dup3}
{\textcolor{blue}{\sc Borowiec, A., W.~Dudek, and S.~Duplij}} (2006).
\newblock Bi-element representations of ternary groups.
\newblock {\em Commun. Algebra\/}~{\bf 34} (5), 1651--1670.

\bibitem[\protect\citeauthoryear{Brieskorn and Saito}{\textcolor{blue}{\sc
  Brieskorn and Saito}}{1972}]{bri/sai}
{\textcolor{blue}{\sc Brieskorn, E. and K.~Saito}} (1972).
\newblock Artin-{G}ruppen und {C}oxeter-{G}ruppen.
\newblock {\em Invent. Math.\/}~{\bf 17}, 245--271.

\bibitem[\protect\citeauthoryear{D\"ornte}{\textcolor{blue}{\sc
  D\"ornte}}{1929}]{dor3}
{\textcolor{blue}{\sc D\"ornte, W.}} (1929).
\newblock Unterschungen \"uber einen verallgemeinerten {G}ruppenbegriff.
\newblock {\em Math. Z.\/}~{\bf 29}, 1--19.

\bibitem[\protect\citeauthoryear{Duplij}{\textcolor{blue}{\sc
  Duplij}}{1998}]{dup18}
{\textcolor{blue}{\sc Duplij, S.}} (1998).
\newblock On semi-supermanifolds.
\newblock {\em Pure Math. Appl.\/}~{\bf 9} (3-4), 283--310.

\bibitem[\protect\citeauthoryear{Duplij}{\textcolor{blue}{\sc
  Duplij}}{2000}]{duplij}
{\textcolor{blue}{\sc Duplij, S.}} (2000).
\newblock {\em Semisupermanifolds and semigroups}.
\newblock Kharkov: Krok.

\bibitem[\protect\citeauthoryear{Duplij}{\textcolor{blue}{\sc
  Duplij}}{2018a}]{dup2018a}
{\textcolor{blue}{\sc Duplij, S.}} (2018a).
\newblock Polyadic algebraic structures and their representations.
\newblock In: {\textcolor{blue}{\sc S.~Duplij}} (Ed.), {\em Exotic Algebraic
  and Geometric Structures in Theoretical Physics}, New York: Nova Publishers,
  pp.\  251--308.
\newblock {arXiv:math.RT/1308.4060}.

\bibitem[\protect\citeauthoryear{Duplij}{\textcolor{blue}{\sc
  Duplij}}{2018b}]{dup2018d}
{\textcolor{blue}{\sc Duplij, S.}} (2018b).
\newblock Polyadic {H}opf algebras and quantum groups, {\it preprint}  Math.
  Inst.,  M\"unster,  57~p., arXiv: math.RA/1811.02712.

\bibitem[\protect\citeauthoryear{Duplij and Marcinek}{\textcolor{blue}{\sc
  Duplij and Marcinek}}{2001}]{dup/mar2001a}
{\textcolor{blue}{\sc Duplij, S. and W.~Marcinek}} (2001).
\newblock Semisupermanifolds and regularization of categories, modules,
  algebras and {Y}ang-{B}axter equation.
\newblock {\em Nucl. Phys. Proc. Suppl.\/}~{\bf 102}, 293--297.
\newblock Int. Conference on Supersymmetry and Quantum Field Theory: D.V.
  Volkov Memorial Conference, Kharkov, Ukraine, July 2000.

\bibitem[\protect\citeauthoryear{Duplij and Marcinek}{\textcolor{blue}{\sc
  Duplij and Marcinek}}{2002}]{dup/mar7}
{\textcolor{blue}{\sc Duplij, S. and W.~Marcinek}} (2002).
\newblock Regular obstructed categories and topological quantum field theory.
\newblock {\em J.~Math. Phys.\/}~{\bf 43} (6), 3329--3341.

\bibitem[\protect\citeauthoryear{Dyer}{\textcolor{blue}{\sc
  Dyer}}{1980}]{dye80}
{\textcolor{blue}{\sc Dyer, J.~L.}} (1980).
\newblock The algebraic braid groups are torsion-free: an algebraic proof.
\newblock {\em Math. Z.\/}~{\bf 172} (2), 157--160.

\bibitem[\protect\citeauthoryear{Grillet}{\textcolor{blue}{\sc
  Grillet}}{1995}]{grillet}
{\textcolor{blue}{\sc Grillet, P.-A.}} (1995).
\newblock {\em Semigroups. An Introduction to the Structure Theory}.
\newblock New York: Dekker.

\bibitem[\protect\citeauthoryear{Hietarinta}{\textcolor{blue}{\sc
  Hietarinta}}{1997}]{hie}
{\textcolor{blue}{\sc Hietarinta, J.}} (1997).
\newblock Permutation-type solutions to the {Y}ang-{B}axter and other
  $n$-simplex equations.
\newblock {\em J.~Phys.\/}~{\bf A30}, 4757--4771.

\bibitem[\protect\citeauthoryear{Howie}{\textcolor{blue}{\sc
  Howie}}{1976}]{howie}
{\textcolor{blue}{\sc Howie, J.~M.}} (1976).
\newblock {\em An Introduction to Semigroup Theory}.
\newblock London: Academic Press.

\bibitem[\protect\citeauthoryear{Hu}{\textcolor{blue}{\sc Hu}}{1997}]{hu97}
{\textcolor{blue}{\sc Hu, Z.-N.}} (1997).
\newblock Baxter-{B}azhanov model, {F}renkel-{M}oore equation and the braid
  group.
\newblock {\em Commun. Theor. Phys.\/}~{\bf 27} (4), 429--434.

\bibitem[\protect\citeauthoryear{Kassel and Turaev}{\textcolor{blue}{\sc Kassel
  and Turaev}}{2008}]{kas/tur2008}
{\textcolor{blue}{\sc Kassel, C. and V.~Turaev}} (2008).
\newblock {\em Braid groups}, Vol. 247 of {\em Graduate Texts in Mathematics}.
\newblock New York: Springer.

\bibitem[\protect\citeauthoryear{Kauffman}{\textcolor{blue}{\sc
  Kauffman}}{1991}]{kauffman}
{\textcolor{blue}{\sc Kauffman, L.~H.}} (1991).
\newblock {\em Knots and Physics}.
\newblock Singapure: World Sci.

\bibitem[\protect\citeauthoryear{Li and Hu}{\textcolor{blue}{\sc Li and
  Hu}}{1995}]{li/hu}
{\textcolor{blue}{\sc Li, Y.~Q. and Z.~N. Hu}} (1995).
\newblock Solutions of {$n$}-simplex equation from solutions of braid group
  representation.
\newblock {\em J. Phys. A. Math. and Gen.\/}~{\bf 28} (7), L219--L223.

\bibitem[\protect\citeauthoryear{Manin and Schechtman}{\textcolor{blue}{\sc
  Manin and Schechtman}}{1989}]{man/sch89}
{\textcolor{blue}{\sc Manin, Y.~I. and V.~V. Schechtman}} (1989).
\newblock Arrangements of hyperplanes, higher braid groups and higher {B}ruhat
  orders.
\newblock In: {\textcolor{blue}{\sc J.~Coates, R.~Greenberg, B.~Mazur, and
  I.~Satake}} (Eds.), {\em Algebraic number theory}, Vol.~17, Boston: Academic
  Press, pp.\  289--308.

\bibitem[\protect\citeauthoryear{Markov}{\textcolor{blue}{\sc
  Markov}}{1945}]{mar45}
{\textcolor{blue}{\sc Markov, A.~A.}} (1945).
\newblock Fundamentals of algebraic theory of braid groups.
\newblock {\em Trudy Mat. Inst. Akad. Nauk SSSR\/}~{\bf 16}, 1--54.

\bibitem[\protect\citeauthoryear{Nikitin}{\textcolor{blue}{\sc
  Nikitin}}{1984}]{nik84}
{\textcolor{blue}{\sc Nikitin, A.~N.}} (1984).
\newblock Semisimple {A}rtinian {$(2,n)$}-rings.
\newblock {\em Vestnik Moskov. Univ. Ser. I Mat. Mekh.\/}~(6), 3--7.

\bibitem[\protect\citeauthoryear{Petrich}{\textcolor{blue}{\sc
  Petrich}}{1984}]{petrich3}
{\textcolor{blue}{\sc Petrich, M.}} (1984).
\newblock {\em Inverse Semigroups}.
\newblock New York: Wiley.

\bibitem[\protect\citeauthoryear{Post}{\textcolor{blue}{\sc Post}}{1940}]{pos}
{\textcolor{blue}{\sc Post, E.~L.}} (1940).
\newblock Polyadic groups.
\newblock {\em Trans. Amer. Math. Soc.\/}~{\bf 48}, 208--350.

\bibitem[\protect\citeauthoryear{Turaev}{\textcolor{blue}{\sc
  Turaev}}{1988}]{tur88}
{\textcolor{blue}{\sc Turaev, V.~G.}} (1988).
\newblock The {Y}ang-{B}axter equation and invariants of links.
\newblock {\em Invent. Math.\/}~{\bf 92} (3), 527--553.

\bibitem[\protect\citeauthoryear{Zamolodchikov}{\textcolor{blue}{\sc
  Zamolodchikov}}{1981}]{zam81}
{\textcolor{blue}{\sc Zamolodchikov, A.~B.}} (1981).
\newblock Tetrahedron equations and the relativistic {$S$}-matrix of
  straight-strings in {$2+1$}-dimensions.
\newblock {\em Commun Math Phys\/}~{\bf 79} (4), 489--505.

\end{thebibliography}
\end{document}